\documentclass[11pt]{article}
\usepackage{amsmath, amsthm, amssymb, algorithm2e, dsfont}
\usepackage{tikz}
\usepackage{fullpage}
\usepackage{enumitem}
\usepackage{hyperref}
\usepackage{xcolor}
\usepackage{url}

\title{Parking on the integers}
\author{Micha{\l} Przykucki\thanks{School of Mathematics, University of Birmingham, Edgbaston, Birmingham, United Kingdom. Supported by the EPSRC grant EP/P026729/1. \newline  E-mail: \texttt{michal.przykucki@gmail.com}.} \and Alexander Roberts\thanks{Mathematical Institute, University of Oxford, Andrew Wiles Building, Radcliffe Observatory Quarter, Woodstock Road, Oxford, United Kingdom. \newline  E-mail: \texttt{\{robertsa, scott\}@maths.ox.ac.uk}.} \and Alex Scott\footnotemark[2] \thanks{Supported by a Leverhulme Trust Research Fellowship.}}

\newtheoremstyle{case}{}{}{\normalfont}{}{shape}{:}{ }{}

\newtheorem{thm}{Theorem}[section]
\newtheorem{lem}[thm]{Lemma}
\newtheorem{prop}[thm]{Proposition}

\newtheorem{cor}[thm]{Corollary}

\theoremstyle{definition}
\newtheorem{defn}[thm]{Definition}

\newtheorem{rem}[thm]{Remark}

\numberwithin{equation}{section}

\newtheoremstyle{case}{}{}{\normalfont}{}{shape}{\normalfont:}{ }{}

\theoremstyle{case}

\def\comment#1{}

\newcommand{\beq}{\begin{eqnarray*}}
\newcommand{\eeq}{\end{eqnarray*}}

\def\build#1_#2^#3{\mathrel{\mathop{\kern 0pt#1}\limits_{#2}^{#3}}}
\newcommand{\beqs}{\begin{eqnarray}}
\newcommand{\eeqs}{\end{eqnarray}}
\newcommand{\eeqsn}{\end{eqnarray*}}                     

\newcommand{\bP}{\mathbb{P}}

\newcommand{\bN}{\mathbb{N}}

\newcommand{\bE}{\mathbb{E}}

\newcommand{\bZ}{\mathbb{Z}}

\newcommand{\cF}{\mathcal{F}}

\newcommand{\indi}{\mathds{1}}

\DeclareMathOperator{\Geom}{Geom}

\DeclareMathOperator{\Bin}{Bin}

\numberwithin{equation}{section}

\begin{document}
\maketitle
\begin{abstract}
Models of parking in which cars are placed randomly and then move according to a deterministic rule have been studied since the work of Konheim and Weiss  in the 1960s.  Recently,  Damron, Gravner, Junge, Lyu, and Sivakoff introduced a model in which cars are both placed and move at random.  Independently at each point of a Cayley graph $G$,  we place a car with probability $p$, and otherwise an empty parking space. Each car independently executes a random walk until it finds an empty space in which to park. In this paper we introduce three new techniques for studying the model, namely the space-based parking model, and the strategies for parking and for car removal. These allow us to study the original model by coupling it with models where parking behaviour is easier to control.  Applying our methods to the one-dimensional parking problem in $\bZ$, we improve on previous work, showing that for $p<1/2$ the expected journey length of a car is finite, and for $p=1/2$ the expected journey length by time $t$ grows like $t^{3/4}$ up to a polylogarithmic factor.
\end{abstract}

\section{Introduction}
\label{sec:intro}

Let $n \geq 1$ and let $P_n$ be a directed path on $[n] = \{1,2,\ldots,n\}$ with directed edges from $i$ to $i-1$ for $i=2,3,\ldots,n.$ Let $1 \leq m \leq n$ and assume that $m$ drivers arrive at vertex $n$ one by one, with the $i$th driver willing to park in vertex $X_i \in [n].$ If the $i$th driver finds $X_i$ empty, they park there. If not, they continue their drive towards $1,$ parking in the first available parking space. If no such spot can be found, the driver leaves the path without parking. We say that $(x_1, \ldots, x_m),$ with $x_1, \ldots, x_m \in [n],$ is a \em parking function \em for $P_n$ if for $X_i = x_i$ for $1 \leq i \leq m,$ all $m$ drivers park on the path.

Parking functions were first studied in the 1960s by \cite{KonheimWeiss-parkingPath}. They evaluated the number of parking functions, which is equivalent to evaluating the probability that an $m$-tuple of independent random variables uniformly distributed on $[n]$ gives a parking function. A similar model, with $P_n$ replaced by a uniform random rooted Cayley tree on $[n]$ was studied by \cite{LacknerPanholzer}. Motivated by finding a probabilistic explanation for some phenomenons observed in \cite{LacknerPanholzer}, \cite{GoldschmidtPrzykucki} analyzed the parking processes on critical Galton-Watson trees, as well as on trees with Poisson(1) offspring distribution conditioned on non-extinction, in both cases with the edges directed towards the root. Note that in all the setups above, drivers have only one choice of route at any time of the process.

In this paper, we are concerned with a related model, introduced by \cite{DGJLS}, in which the cars move at random. Let $L=(V,E)$ be a Cayley graph on a group $V$ with generating set $R$, and let  $\mu$ be a probability distribution on $R$: we will refer to such a triple $(L,R,\mu)$ as a {\em parking triple}. At time $0$, each position $v \in V$ is independently assigned a car with probability $p$ 
or a parking space with probability $1-p$. The cars follow independent random walks with increments  $\mu$ 
and each car continues to follow the random walk until it finds a free space where it parks (if more than one car arrives at a free space at the same time, then one is chosen to park according to some rule).\footnote{We note that \cite{DGJLS} work in a slightly more general setting, see Section 2 in \cite{DGJLS}. While our results in Section \ref{sec:definitions} also hold in the setting in \cite{DGJLS}, we believe that the class of Cayley graphs is a fairly general setting and the link with lattices is a little clearer.}

There is a wide range of well studied models with a flavour similar to parking funtions. In particular we note various gas particle models which track the movement and annihilation of particles in a system. Of particular note are \em annihilating random walks \em (see \cite{EN74} or \cite{Arr83} for example) where particles move according to some random path and annihilate upon collision with another particle; and \em two-type diffusion limited annihilating systems \em where particles of two types move according to some random path and annihilate upon collision with a particle of the other type (see \cite{BL91} or \cite{JJLS} for example). Results analogous to those proven here are well known in both of these settings.

We are interested in the distribution of journey lengths of cars. We introduce the stopping time $\tau^v$ where $\tau^v = 0$ if position $v$ is a parking space, and otherwise $\tau^v$ is the time the car starting at $v$ takes to park ($\tau^v= \infty$ if the car never parks). We also write $\tau = \tau^0$ (by symmetry we only need to consider $v=0$). Given $t \ge 0$ and a vertex $v,$ let
\[
V^v(t)= \left|\left\{(u,s) \in V \times [t] : \mbox{car $u$ visits $v$ at time $s$}\right\}\right| + \indi_{\{v \mbox{ \small is a car}\}}
\]
be the number of cars that visit $v$ up to time $t.$

In the particular case of the lattice ${\mathbb Z}^d$ (with edges joining lattice points at Euclidean distance 1), \cite{DGJLS} prove the following theorem.
\begin{thm}\label{thm:DGJLS}
Consider the parking process on $\bZ^d$ with simple symmetric random walks.
\begin{enumerate}
 \item If $p \geq 1/2$ then  $\bE[\tau] = \infty$ with $\bE[\min\{\tau,t\}] = (2p-1)t + o(t).$
 \item If $p < 1/2$ then $\tau$ is almost surely finite. Moreover, if $p < (256d^6e^2)^{-1}$ then $\bE[\tau] < \infty.$
\end{enumerate}
\end{thm}

For $p > 1/2,$ Theorem \ref{thm:DGJLS} gives good asymptotics for $\bE[\min\{\tau,t\}].$ However, for $p = 1/2$ Theorem \ref{thm:DGJLS} only tells us that $\bE[\min\{\tau,t\}]$ is $o(t),$ while following the seminar by \cite{Junge-seminar} we know that the authors of \cite{DGJLS} conjecture that for $d=1$ and $p=1/2$ we have $\bE[\min\{\tau,t\}] = \Theta(t^{3/4})$. Moreover, for $d=1$,  Theorem \ref{thm:DGJLS} only gives $\bE[\tau] < \infty$ for $p < 0.000528,$ while it is conjectured that this holds for all $p < 1/2.$

Here, we address both conjectures of the authors of \cite{DGJLS} when $d=1,$ and prove the following two theorems. The first considers the parking problem on $\bZ$ with $p=1/2$ where we give strong bounds on the asymptotic growth of $\bE[\min\{\tau,t\}]$ by showing that it indeed equals $t^{3/4}$ up to a fractional power of $\log t.$

\begin{thm}\label{thm:p=1/2}
For the parking problem on $\bZ,$ when $p=1/2,$ there exist constants $C, c > 0$ such that
\[
 c t^{3/4} (\log t)^{-1/4} \leq \bE[\min\{\tau,t\}] \leq C t^{3/4}.
\]
\end{thm}

\begin{rem}\label{boe}
Through the expression
	\begin{align*}
		\bE[\min\{\tau,t\}] = \sum_{s=0}^{t-1}\bP[\tau>s],
	\end{align*}
Theorem \ref{thm:p=1/2} gives bounds for the tail $\bP[\tau > t]$, i.e. Theorem \ref{thm:p=1/2} implies the existence of constants $C',c'>0$ such that
\[
c' t^{-1/4} (\log t)^{-1/4} \leq \bP[\tau > t] \leq C' t^{-1/4}.
\]
This interpretation additionally allows us to present some heuristics for the conjectured exponent $3/4$.

Consider an interval $[-C t^{1/2}, C t^{1/2}]$ around 0, for some $C > 0$ large. By the properties of the simple symmetric random walk, we would not expect too many cars starting in that interval to exit it by time $t$. By the properties of the binomial distribution, with some uniformly positive probability we can also expect at least $C t^{1/2} + t^{1/4}$ cars to start in this interval. Since every car that parks occupies exactly one parking space, and we only have $C t^{1/2} - t^{1/4}$ such spaces to start with, that surplus of $2t^{1/4}$ cars will not find a parking space by time $t$, consequently suggesting that $\bP[\tau>t] \approx t^{1/4} / t^{1/2} = t^{-1/4}$, and this corresponds to an approximate guess of $\bE[\min\{\tau,t\}] \approx \Theta(t^{3/4})$. In the proof of the lower bound in Theorem \ref{thm:p=1/2} we will follow a similar line of thought, filling in the gaps left in the intuitions above.
\end{rem}

The second theorem considers the parking problem on $\bZ$ with $p<1/2$ where we confirm that the expected journey length of a car is finite as predicted. \cite{DGJLS} also ask whether (for a large family of parking processes)  there is a critical exponent $\gamma >0$ such that, for some constant $C > 0$, $\bE[\tau] \sim C\left(1/2 -p\right)^{-\gamma}$ as $p$ increases to $1/2.$ For the parking problem on $\bZ$, we have a partial result in this direction.
\begin{thm}\label{thm:p<1/2}
For the parking problem on $\bZ,$ when $p<1/2$ we have $\bE[\tau] < \infty.$ Moreover $\bE[\tau] = O\left((1/2 - p)^{-6}\right)$ as $p\nearrow1/2.$
\end{thm}

We remark that since the completion of this work, \cite{JJLS} showed that the critical exponent is $-3+o(1)$ (see Theorems 3 and 4 in \cite{JJLS}) for the continuous time-base parking problem where cars move at exponentially distributed times. The proof they give involves a coupling where they release cars one-by-one, seeing if they park within the first $T$ moves, and finally asserting that it takes on average time $T$ for a car to move $T$ times. As such, their proof also applies to the discrete time-base parking problem analysed in this paper with the last step no longer necessary.

In this paper, we will consider strategies that modify the car-parking process. 
We will introduce two types of strategy: {\em parking strategies} where we allow cars to choose whether or not to park in an available space, and {\em car removal strategies} where we remove cars from the parking process (we defer formal definitions to Section \ref{sec:definitions}). In each case the strategies will be previsible in the sense that no future information may be used when choosing whether or not a car parks at a particular point in time. For a parking triple $(L,R,\mu)$ and a strategy $S$, we will write $V^v_S(t)$ for the value of $V^v(t)$ when strategy $S$ is followed, and similarly $\tau^v_S$; we write $G$ for the greedy strategy (i.e.~the original process).

The key properties of parking and car removal strategies that we shall use are given in the following theorems, which show that no parking strategy is quicker than the greedy one, and that adding car removal makes parking easier. 
We note that these results hold in the more general setting of Cayley graphs.

\begin{thm}\label{thm:supermarket}
Let $S$ be a parking strategy on the parking triple $(L,R,\mu)$. Then for all $t, k \ge 0$ and vertices $v$,
\[
\bP\left[V^v_G(t) \le k\right] \ge \bP\left[V^v_S(t) \le k\right].
\]
\end{thm}

\begin{thm}
\label{thm:pinkfloyd}
Let $Q$ be a car removal strategy on the parking triple $(L,R,\mu)$. Then $\tau^v_Q \le \tau^v_G$, and for all $t \ge 0$ we have $V^v_Q(t) \le V^v_G(t)$.
\end{thm}

In order to prove Theorem \ref{thm:supermarket} we introduce a new construction of the model, in which the cars follow directions stored at the vertices they visit, rather than their own individual random walks.  We will refer to this as the {\em space-based model}, in contrast to the {\em car-based model} described above.  Even though the stochastic properties of the two models are equivalent, the new model allows us to control the quantity $V^v(t)$ better, and we are then able to easily deduce the desired result for the original parking problem.

The paper is organised as follows. In Section \ref{sec:definitions} we define the parking processes, introduce the notions of parking strategies and car removal, and prove Theorems \ref{thm:supermarket} and \ref{thm:pinkfloyd}. This allows us to consider both more and less restrictive parking problems, which we use in our arguments. In Section \ref{sec:probabilities} we recall some known probability bounds that are used in this paper. In Section \ref{sec:p=1/2upper} we prove the upper bound on $\bE[\min\{\tau,t\}]$ in Theorem \ref{thm:p=1/2}, and in Section \ref{sec:p=1/2lower} we prove the lower bound. In Section \ref{sec:p<1/2} we prove Theorem \ref{thm:p<1/2}. Finally in Section \ref{yaymorewaffle} we conclude the paper with some related problems and open questions.

Throughout this paper, we use the notation $a \wedge b = \min\{a,b\}.$ For a normally distributed random variable $Z$ with mean $0$ and variance $1,$ we write $\Phi(x) = \bP\left[Z \le x\right].$

\section{Model specifics, parking strategies, and car removal}
\label{sec:definitions}

We will want to consider slight modifications of the original parking problem on $\bZ.$ In this section, we introduce new notation for these modifications and also compare these modifications to the original problem. The first modification is the addition of \emph{parking strategies}. The second is the addition of \emph{car removal} to the process. We compare the expected journey length of a car by time $t,$ showing that non-trivial parking strategies increase expected journey times while car removal decreases them. In fact, we are able to show that these bounds hold for any parking triple.

\subsection{The car-based parking model}

Let us recall some definitions.
Let $H$ be a group and $R$ be a generating set for $H$. The \em Cayley graph \em of $H$ with respect to $R$ is the edge-coloured directed graph $L = (H,E)$ where $$E := \{(h,hr) : h \in H, r \in R\},$$ and the edge $(h,hr)$ is coloured $r$.
Note that if $R$ is closed under taking inverses then $(x,y) \in E$ if and only if $(y,x) \in E,$ and so we can just consider the underlying graph. For example, the $d$-dimensional integer-lattice $\bZ^d$ can be thought of as the abelian group with generating set $\{e_1,-e_1,\ldots,e_d,-e_d\} \subset \bZ^d$ where the $i$-th co-ordinate of $e_i$ is $1$ and all others are $0$.

A {\em parking triple} is an ordered triple $(L,R,\mu)$ , where
$L=(V,E)$ is a Cayley graph on a group $V$ with generating set $R$ and $\mu$ is a probability distribution on $R$ 
 (In later sections we will be interested in the parking problem on $\bZ$, namely the parking triple $(\bZ,\{-1,+1\},\mu^{\bZ})$ where $\mu^{\bZ}(-1) = \mu^{\bZ}(+1) = 1/2$.  However, the results in this section hold in the more general model.)

We define the parking problem on the parking triple $(L,R,\mu)$ as follows.

\begin{defn}\label{defn:free parking}
Independently for each vertex $v \in V$, let: 
\begin{itemize}
\item $X^v = (X^v_0,X^v_1,\ldots)$ be a Markov chain on $L$ with $X^v_0=v$, and transition matrix $(p_{u,w})$ where $p_{u,ur} = \mu(r)$ for each $u\in V$ and $r\in R$, and $p_{u,w} = 0$ otherwise.

\item $(U^v_s)_{s \in \bN}$ be a sequence of independent $\mathrm{Unif}([0,1])$ random variables. 

\item $B^v$ be a $\mbox{Bernoulli}(p)$ random variable. We initially place a car at $v$ when $B^v = 1$ and otherwise a parking space with the capacity for one car.
\end{itemize}

A car starting at vertex $v$ moves according to the Markov chain $X^v$ until it finds a free parking space and parks there. (We do not use the random walks $X^v$ for those $v$ where we initially place a parking space; we define them just for the simplicity of the model.) If cars $v_1,\ldots,v_k$ all arrive at the same free parking space at time $s,$ we park car $v_j$ with smallest $U^{v_j}_s.$
\end{defn}

We shall sometimes refer to the model in Definition \ref{defn:free parking} as the \emph{car-based parking model}. We remark that we generate new independent tie-splitting values (the $(U^v_s)_{v\in V}$) for each $s$ to maintain fairness. Indeed, had we relied on a single value $U_v$ for a car $v$ throughout the whole history of the parking process, the cars that had encountered more cars and lost the tie-splits initially would be more likely to keep losing them, and consequently not parking, later in the process.

Let $(\Omega,\cF,\bP)$ be a probability space. A \emph{filtration} is a sequence $\cF_0 \subseteq \cF_{1} \subseteq  \ldots$ of $\sigma$-algebras. A random variable $\tau : \omega \rightarrow \bN$ is a \emph{stopping time} with respect to a filtration $(\cF_{t})^{\infty}_{t=0}$ if $\tau^{-1}(\{t\}) \in \cF_{t}$ for each $t \in \bN.$  In the car-based model, for the parking problem on the parking triple $(L,R,\mu)$, we consider the probability space $(\Omega,\cF,(\cF_t)_{t\ge0},\bP)$, where an elementary event $\omega \in \Omega$ is of the form $\omega = \left((B^v)_{v\in V}, (X^v_s)_{v\in V, s\in \bN},(U^v_s)_{v\in V,s\in \bN}\right)$, and the filtration $\left(\cF_t\right)_{t \ge 0}$ is defined by
\[
\cF_t = \sigma((B^v)_{v \in V}, (X^v_s)_{v \in V,0 \leq s \leq t}, (U^v_s)_{v \in V,1 \leq s \leq t})
\]
for all $t \geq 0.$

\subsection{Parking strategies and the space-based model}

In the model we have defined, all cars try to park as soon as they reach a free parking space. This can be thought of as a \em parking strategy\em. Let $G$ denote this ``greedy" parking strategy: a car parks as soon as it can. It will be useful to consider different (possibly random) parking strategies as a way of controlling where cars park. In the definition below we introduce parking strategies more formally; $S_t(v,w)=1$ should be thought of as the event that the car starting from $v$ parks in $w$ at time $t.$

\begin{defn}\label{parkdefn}
Let $(L,R,\mu)$ be a parking triple. A \emph{parking strategy} $S = (S_t(v,w))_{t\ge 1, v,w \in V}$ for the car-based model on $(L,R,\mu)$ is a sequence of random variables taking values in $\{0,1\}$ with the following properties:
\begin{itemize}
\item $S_t(v,w)$ is $\cF_t$-measurable for each $v,w \in V$ and $t \ge 1.$
\item $\sum_{t\ge 1, w \in \bZ} S_t(v,w) \le 1$ (a car parks at most once).
\item $\sum_{t \ge 1, v \in \bZ} S_t(v,w) \le 1$ (a parking space can hold only one car).
\item $S_t(v,w) = 0$ whenever $B^v = 0$ (a parking space cannot be filled by a non-existent car).
\item $S_t(v,w) = 0$ whenever $B^w = 1$ (a car cannot park where there is no parking space).
\item $S_t(v,w) = 0$ whenever $X^v_t \neq w$ (a car cannot park in a space which is not its current position).
\end{itemize}
A car starting at $v$ \em parks in space $w$ at time $t$ \em if and only if $S_t(v,w) = 1.$
\end{defn}
Note that $S_t(v,w)$ being $\cF_t$-measurable means that our parking strategy is previsible, and that the parking time of a car is a stopping time.

For a parking strategy $S$ and an event $E$ we let $\bP^S\left[E\right]$ denote the probability of $E$ when all cars follow strategy $S$ (note that $\bP = \bP^G$). We will also allow random parking strategies, which require suitable adjustments to the $\sigma$-algebra and the filtration (for example, we may independently flip a coin at the start and choose different parking strategies depending on whether the coin is heads or tails).

Equipped with these new definitions, we are nearly ready to prove Theorem \ref{thm:supermarket}. The final element we shall need is a stochastically equivalent parking process, where the moves of cars are attached to spaces rather than the cars; we shall refer to this model as the \emph{space-based parking model}.

\begin{defn}\label{defn:space parking}
Let $(L,R,\mu)$ be a parking triple. Independently for each vertex $v \in V$, let: 
\begin{itemize}
\item $(E^v_n)_{n \in \bN}$ be a sequence of independent $\mu$-random variables,

\item $(\tilde{U}^v_s)_{s \in \bN}$ be a sequence of independent $\mathrm{Unif}([0,1])$ random variables. 

\item $\tilde{B}^v$ be a $\mbox{Bernoulli}(p)$ random variable. We initially place a car at $v$ when $\tilde{B}^v = 1$ and otherwise a parking space with the capacity for one car.
\end{itemize}

When a single car arrives (but does not park) at position $v$, it leaves in the next time step according to the first unused $E^v_n$. If the set of cars $\{w_1,\ldots,w_r\}$ arrives at $v$ at time $s$ and do not park, they collect the next $r$ unused directions $E^v_n, E^v_{n+1}, \ldots, E^v_{n+r-1}$, in the order determined by their increasing values of $U^{w_{\ell}}_s$.
\end{defn}

For the space-based parking model on the parking triple $(L,R,\mu)$ it is less obvious what the the probability space $(\tilde{\Omega},\tilde{\cF},(\tilde{\cF}_t)_{t\ge0},\tilde{\bP})$ should be. This is because the number of directions collected from $E^v$ by cars that visit $v$ by time $t$ but do not park there depends on the behaviour of cars starting at distance at most $t$ from $v$ in the first $t$ steps of the process. Hence, we can define the filtration $\left(\tilde{\cF}_t\right)_{t \ge 0}$ to be
\[
\tilde{\cF}_t = \sigma((\tilde{B}^v)_{v \in V}, (E^v_n)_{(v,n) \in \{V\times \bZ^+: n \le D^v(t)\}}, (\tilde{U}^v_s)_{v \in V,1 \leq s \leq t})
\]
for all $t \geq 0$, where
\[
D^v(t)= \left|\left\{(u,s) \in V \times \{0,\ldots,t-1\} : \mbox{car $u$ visits $v$ at time $s$ but does not park there}\right\}\right|
\]
is the number of departures from $v$ before time $t$. In other words $\tilde{\cF}_t$ contains exactly the parts of $(B,E,U)$ which determine the movements of cars up to time $t$, including which cars have parked. Although at first sight the Reader might find the random nature of $\tilde{\cF}_t$ confusing, we hope that it will not cause difficulties when following the proofs.

\begin{defn}\label{spaceparkdefn}
Let $(L,R,\mu)$ be a parking triple. A \emph{parking strategy} $\tilde{S} = (\tilde{S}_t(v,w))_{t\ge 1, v,w \in V}$ for the space-based model on $(L,R,\mu)$ is a sequence of random variables taking values in $\{0,1\}$ with the following properties:
\begin{itemize}
\item $\tilde{S}_t(v,w)$ is $\tilde{\cF}_t$-measurable for each $v,w \in V$ and $t \ge 1$.
\item $\sum_{t\ge 1, w \in \bZ} \tilde{S}_t(v,w) \le 1$ (a car parks at most once).
\item $\sum_{t \ge 1, v \in \bZ} \tilde{S}_t(v,w) \le 1$ (a parking space can hold only one car).
\item $\tilde{S}_t(v,w) = 0$ whenever $\tilde{B}^w = 1$ (a car cannot park where there is no parking space).
\item $\tilde{S}_t(v,w) = 0$ whenever $\tilde{B}^v = 0$ (a parking space cannot be filled by a non-existent car).
\item For all $v \in L$ such that:
\begin{itemize}
\item $\tilde{B}^v = 1$, and
\item for all $u \in L$ and $s \leq t-1$ we have $\tilde{S}_s(v,u) = 0$,
\end{itemize}
let $E^{v_1}_{n_1}, E^{v_2}_{n_2}, \ldots, E^{v_t}_{n_t}$ be the directions selected by $v$ in the first $t$ steps of its walk (note that we have $v_1 = v$). Then $\tilde{S}_t(v,w) = 0$ if the walk obtained by starting at $v$ and following these directions does not end at $w$
(a car cannot park in a space which is not its current position).
\end{itemize}
A car starting at $v$ \em parks in space $w$ at time $t$ \em if and only if $\tilde{S}_t(v,w) = 1.$
\end{defn}

We let $\tilde{G}$ denote the greedy parking strategy in the space-based model. In the following proposition we show that parking strategies in the car-based parking process are stochastically equivalent to corresponding parking strategies in the space-based parking proces.

\begin{prop}
\label{prop:sameOld}
Let $(L,R,\mu)$ be a parking triple. Let $S$ and $\tilde{S}$ be parking strategies for the car-based model and the space-based model on $(L,R,\mu)$ respectively, and assume that for all $t \geq 1$ and $v, w \in L$ we have $S_t(v,w) = \tilde{S}_t(v,w)$ whenever the following conditions hold:
\begin{enumerate}
\item $B^v = \tilde{B}^v$ for all $v \in L$ (the same cars appear in both models),
\item for all $1 \leq s < t$ and $v, w \in L$ we have $S_s(v,w) = \tilde{S}_s(v,w)$ (at every time $1 \leq s < t$, the same cars park in the same parking places in both models), and 
\item for all $1 \leq s \leq t$, every car that does not park before time $s$, occupies the same position at time $s$ in both models
\end{enumerate}
(i.e. the strategies $S$ and $\tilde{S}$ behave identically whenever the cars behave identically up to time $t$ in the two processes). Then for any two sets $X \subset L \times L \times \bN, Y \subset L$, and the event
\[
 A_{X, Y} = [ \mbox{ for all } (v_i, w_i, t_i) \in X, \mbox{car } v_i \mbox{ is in } w_i \mbox{ at time } t_i; \mbox{ for all } w_j \in Y, w_j \mbox{ is a parking space } ]
\]
we have $\bP^S[A_{X,Y}] = \tilde{\bP}^{\tilde{S}}[A_{X,Y}]$.
\end{prop}

\begin{proof}
We have $\bP^S[A_{X,Y}], \bP^{\tilde{S}}[A_{X,Y}] \leq (1-p)^{|Y|}$, so if $|Y| = \infty$ then $\bP^S[A_{X,Y}], \bP^{\tilde{S}}[A_{X,Y}] = 0$ and the proposition holds.

If $|X| = \infty$ then $A_{X,Y}$ must either describe the moves of infinitely many cars, or there must be a car $v$ such that $A_{X,Y}$ gives the position of $v$ at infinitely many times, or there are some $w_1 \neq w_2$ and some $v \in L, t \in \bN$, such that $(v, w_1, t), (v, w_2, t) \in X$. In all of these cases we have $\bP^S[A_{X,Y}], \bP^{\tilde{S}}[A_{X,Y}] = 0$.

Hence we can assume that $|X|, |Y| < \infty$. Then, let
\[
U = \{v : (v,w,t) \in X \} \cup \{w : (v,w,t) \in X \} \cup Y,
\]
and let $T = \max \{t : (v,w,t) \in X\}$. Then, in the car-based model, we can express $A_{X,Y}$ as a finite union of finite events concerning the variables $B^v, X^v_t, U^v_t$, for $t \leq T$ and $v$ at distance at most $T$ from some element in $U$, describing the car/parking space status and the step-by-step moves of cars in the $T$-neighbourhood of the elements if $U$. Analogously, in the space-based model, we can express $A_{X,Y}$ as a finite union of finite events concerning the variables $\tilde{B}^v, E^v_n, \tilde{U}^v_t$, for $t \leq T$, $n \leq T^2$, and $v$ at distance at most $T$ from some element in $U$. The proposition now follows from the properties of $S$ and $\tilde{S}$, from the identical distributions and independence of $(B^v)_{v \in V}$ and $(\tilde{B}^v)_{v \in V}$, of the $(U^v_t)_{v \in V,t \geq 0}$ and $(\tilde{U}^v_t)_{v \in V,t \geq 0}$, as well as of $(X^v)_{v \in \bZ}$ and $((E^v_n)_{n \in \bN})_{v \in \bZ}$ (observe that each of $E^v_n$ is used at most once in the process).
\end{proof}

Proposition \ref{prop:sameOld} will allow us to deduce Theorem \ref{thm:supermarket} from the following theorem.

\begin{thm}\label{thm:supermarketSpaces}
Let $\tilde{S}$ be a parking strategy for the space-based parking process on the parking triple $(L,R,\mu)$. For a vertex $v$, we write $V^v_{\tilde{S}}(t)$ for the value of $V^v(t)$ when strategy $\tilde{S}$ is followed, and $V^v_{\tilde{G}}(t)$ for the value of $V^v(t)$ when the greedy strategy is followed. Then for all $t \ge 0,$ we have $V^v_{\tilde{S}}(s) \geq V^v_{\tilde{G}}(s)$.
\end{thm}

\begin{proof}
Consider the space-based parking process on a parking triple $(L,R,\mu)$. Let $T^{v,r^{-1}}(t-1)$ be the number of cars that arrived at $vr^{-1}$ in the first $t-1$ time steps and then picked up $E^{vr^{-1}}_n = r$. Observe that $V^v(t)$ is equal to the sum over $r\in R$ of $T^{v,r^{-1}}(t-1)$, plus $1$ if a car started at $v$ initially. By induction on $t$ we prove the following claim: for all $t \geq 0$ we simultaneously have $T^{v,r^{-1}}_{\tilde{S}}(t-1) \geq T^{v,r^{-1}}_{\tilde{G}}(t-1)$ and $V^v_{\tilde{S}}(t) \ge V^v_{\tilde{G}}(t),$ for all $r \in R$ (where again $T_{\tilde{S}}$ and $T_{\tilde{G}}$ denote the quantities when all cars follow strategy ${\tilde{S}}$ or ${\tilde{G}}$ respectively).

If a car parks at $v$ in the first $t$ time steps under ${\tilde{S}}$ then $v$ must have initially been a parking space; then, if at least one car drove to $v$ under ${\tilde{G}}$, it follows that some car parked in $v$ under ${\tilde{G}}$ as well. Hence if the number of cars arriving at any vertex in the first $t$ time steps is at least as large under ${\tilde{S}}$ as under ${\tilde{G}}$, the same applies to the number of cars leaving $v$ in the first $t+1$ time steps. Moreover, for each $r \in R,$ since the directions $E^{vr^{-1}}_n$ are selected one-by-one in a fixed order, $V^{vr^{-1}}_{\tilde{S}} (t) \ge V^{vr^{-1}}_{\tilde{G}} (t)$ implies $T^{v,r^{-1}}_{\tilde{S}} (t) \geq T^{v,r^{-1}}_{\tilde{G}} (t).$

The base case $t = 0$ of the induction is trivial. Hence suppose that our claim is true for $t = s -1 \ge 0.$  By induction, for each $r \in R,$ we have $V^{vr^{-1}}_{\tilde{S}}(s-1) \ge V^{vr^{-1}}_{\tilde{G}}(s-1)$; hence we have $T^{v,r^{-1}}_{\tilde{S}}(s-1) \geq T^{v,r^{-1}}_{\tilde{G}}(s-1).$ We then obtain
\begin{align*}
V^v_{\tilde{S}}(s) & = \sum_{r\in R} T^{v,r^{-1}}_{\tilde{S}}(s-1) + \indi_{\{v \mbox{ \small is a car}\}} \\
 & \geq \sum_{r\in R} T^{v,r^{-1}}_{\tilde{G}}(s-1) + \indi_{\{v \mbox{ \small is a car}\}} \\
 & = V^v_{\tilde{G}}(s).
\end{align*}
This completes the proof of Theorem \ref{thm:supermarketSpaces}.
\end{proof}

\begin{rem}
We observe that this pathwise dominance in the space-based parking process does not in general hold for the car-based parking process. Suppose that cars are initially only located at $\{1, 2, 4\}$, that cars $2$ and $4$ always go left, and that the path of car $1$ is $+1, +1, -1, +1, -1, +1, \ldots$, i.e., two steps to the right followed by an infinite sequence of pairs $-1, +1$. In the greedy strategy, car $2$ parks at $0$ and car $4$ parks at $3$, while $1$ never parks alternating between positions $2$ and $3$ forever. Consequently we have $V^2_{\tilde{G}}(s), V^3_{\tilde{G}}(s) \approx s/2$. On the other hand, if we ban all cars from parking on the first time steps, car $2$ still parks at $0$, car $4$ parks at $-1$, and $1$ parks at $3$, and all vertices are visited only finitely many times.
\end{rem}

\begin{proof}[Proof of Theorem  \ref{thm:supermarket}] 
Let $S$ be a parking strategy for the car-based model on the parking triple $(L,R,\mu)$, let $v \in L$, and let $t, k \ge 0$. Observe that for parking strategies in the space-based model, the filtration $\tilde{\cF}_t$ carries all the information about the moves of all cars up to time $t$. Therefore we can design a parking strategy $\tilde{S}$ for the space-based model, such that the assumptions of Proposition \ref{prop:sameOld} are satisfied for $S$ and $\tilde{S}$.

Next, we can express the event $[V^v_S(t) \le k]$ as a finite union of events $A_{X,Y}$, defined as in Proposition \ref{prop:sameOld}, describing the car/parking space status and movements of cars starting at distance at most $t$ from $v$, such that at most $k$ cars arrive at $v$ by time $t$ under $S$. By Proposition \ref{prop:sameOld}, we have $\bP^S[A_{X,Y}] = \tilde{\bP}^{\tilde{S}}[A_{X,Y}]$. By Theorem \ref{thm:supermarketSpaces} we have $V^v_{\tilde{S}}(s) \geq V^v_{\tilde{G}}(s)$ deterministically, hence if $A_{X,Y} \subseteq [V^v_{\tilde{S}}(t) \le k]$, then also $A_{X,Y} \subseteq [V^v_{\tilde{G}}(t) \le k]$. Thus we have $\bP[V^v_{\tilde{G}}(t) \le k] \geq \tilde{\bP}[V^v_{\tilde{S}}(t) \le k]$, and since by applying Proposition \ref{prop:sameOld} again we find that $\bP^G[A_{X,Y}] = \tilde{\bP}^{\tilde{G}}[A_{X,Y}]$, we finally obtain $\bP[V^v_G(t) \le k] \geq \bP[V^v_S(t) \le k]$ as claimed.
\end{proof}

In the rest of this paper, we shall consider the car-based parking model only. We remark that for parking times we may not make a conclusion similar to Theorem \ref{thm:supermarket}. For example, consider the parking strategy where all but one car is instructed to never park. The chosen car will have a much easier job of finding a parking space. To combat this, we need some symmetry that will allow us to compare visits to a space and parking times of cars, and therefore make use of Theorem \ref{thm:supermarket}

We say that a parking strategy $S$ on the parking triple $(L,R,\mu)$ is \emph{weakly translation invariant} if for all $v,w \in V$, $r\in R$  and $t \ge 0$, 
\[
 \bP^S\left[S_t(v,w)=1\right] = \bP^S\left[S_t(vr,wr)=1\right].
\]
An equivalent property is that for all $v,w \in V$, $r\in R$ and $t \geq 1$,
\[
\bP^S\left[\mbox{car $v$ arrives at spot $w$ at time $t$}\right] = \bP^S\left[\mbox{car $vr$ arrives at spot $wr$ at time $t$}\right].
\] 

\begin{rem}
This is a rather weak form of translation invariance -- it does not control joint events in any sense. Since in this paper we are predominantly working with expectations, we do not need to worry about this. A more natural form of translation invariance is the following form: a parking strategy $S$ on the parking triple $(L,R,\mu)$ is \emph{strongly translation invariant} if for any $r\in R,$ the probability measure $\bP$ is invariant with respect to a translation by $r$. (The same is true for car removal strategies which we introduce later.) We note that the parking strategy (respectively, car removal strategy) we use in Section \ref{sec:p=1/2upper} (respectively, Section \ref{sec:p=1/2lower}) are in fact strongly translation invariant.
\end{rem}

Weak translation invariance allows us to equate car journey lengths with total number of visits to a position in $V.$

\begin{lem}\label{lem:dual}
Let $S$ be a weakly translation invariant strategy on the parking triple $(L,R,\mu)$. Then for all $t \ge 0$ and $v \in V,$
\[
 \bE^S[\tau \wedge t] = \bE^S[V^v(t)].
\]
\end{lem}

One can observe that the equality in Lemma \ref{lem:dual} only holds for expectation, since these are very distinct random variables: for instance, $\tau \wedge t$ is bounded by $t$, while $V^v(t)$ could reach the order $t^2$.

We remark that Lemma \ref{lem:dual} is a special case of the well known and more general Mass-Transport principle \cite[Theorem 8.7]{L+P} and a similar result was noted at \cite[Lemma 4.1]{DGJLS}. Since the proof is very short in our setting, we include it for self-containment.

\begin{proof}
Let $t \ge 0$ and fix an arbitrary $v \in V.$ Write $B_t(v)$ for the vertices of $L$ connected to $v$ by a path of length at most $t$. By translation invariance
\begin{align*}
 \bE^S[\tau \wedge t] = \bE^S[\tau^v \wedge t] & = \sum_{s \in [t]}\sum_{w \in B_{t}(0)} \bP^S\left[\mbox{car $v$ arrives at spot $vw$ at time $s$}\right] \\
 & = \sum_{s \in [t]}\sum_{w \in B_t(0)} \bP^S\left[\mbox{car $vw^{-1}$ arrives at spot $v$ at time $s$}\right] \\
 & = \bE^S[V^v(t)].
\end{align*}
\end{proof}

The following easy corollary of Theorem \ref{thm:supermarket} and Lemma \ref{lem:dual} is crucial for our arguments, and considers the expected journey of a car up to time $t$ under different parking strategies. It will allow us to derive upper bounds on $\bE^G[\tau \wedge t]$ by considering a different parking strategy which is easier to control.

\begin{cor}
\label{lem:supermarket}
Let $S$ be a weakly translation invariant parking strategy on the parking triple $(L,R,\mu)$. Then for all $t \ge 0,$
\[
 \bE^S[\tau \wedge t] \ge \bE^G[\tau \wedge t].
\]
\end{cor}

\subsection{Car removal strategies}

Another way to modify the car parking problem is through \emph{car removal strategies}. Under certain circumstances it will be helpful to pretend that a car has been removed from the process. A car is removed during a step, and it is parked off $V$. So if car $v$ is at position $w$ at time $t,$ and is removed during step $t+1$, we remove the car from the process without it taking up a parking space and set $\tau^v = t+1.$ We remark that we will always assume a greedy parking strategy when we have a non-trivial car removal strategy.

\begin{defn}\label{teledefn}
Let $(L,R,\mu)$ be a parking triple. A \emph{car removal strategy} $Q = (Q_t(v))_{t\ge 1, v\in V}$ on $(L,R,\mu)$ is a sequence of random variables taking values in $\{0,1\}$ with the following properties:
\begin{itemize}
\item $Q_t(v)$ is $\cF_t$-measurable for each $v \in V$ and $t \ge 1.$
\item $Q_t(v) = 0$ whenever $B^v = 0$ (a non-existent car cannot be removed).
\item $\sum_{t\ge 1} Q_t(v) \le 1$ (a car can only be removed once).
\end{itemize}
A car starting at $v$ is removed in the $t$-th time step if and only if $Q_t(v) = 1.$
\end{defn}
As we did for parking strategies, we define $\bP^Q$ for a car removal strategy $Q$. Whenever we explicitly consider a process involving car removal strategies, we assume that all vehicles follow the greedy parking strategy.

We are now ready to prove Theorem \ref{thm:pinkfloyd}.  In the one-dimensional setting this will allow us to derive lower bounds on $\bE[\tau \wedge t]$ by considering an interval and removing cars that enter or leave the interval.

\begin{proof}[Proof of Theorem \ref{thm:pinkfloyd}]
For each $w \in V$ and $t \ge 0,$ let $W^w_Q(t)$ be the set of unparked cars at position $w$ at time $t$ under $Q$, and let $W^w_G(t)$ denote the same quantity under $G$ (recall that under $G$, which is the greedy parking strategy, there is no car removal). We start by showing that at every position $w \in V$ and for every time $t \ge 0$ we have $W^w_Q(t) \subseteq W^w_G(t)$. We prove this by induction on $t \ge 0$. The base case $t=0$ is trivial, hence suppose that the claim is true up to and including time $t-1$.

Fix a position $w$ and observe first that if a parking space $w$ is filled at time $t$ under $Q$, then a car $v$ from $W^{wr^{-1}}_Q(t-1)$ must arrive at $w$ at time $t$ for some $r \in R$. By the inductive hypothesis, $v$ must be in the appropriate set in $W^{wr^{-1}}_G(t-1)$, and so it must arrive at $w$ at time $t$ under $G$ (note that we are in the original parking process, where cars have random walks attached to them, rather than the space-based parking process considered in the proof of Theorem \ref{thm:supermarketSpaces}). Therefore under $G$ either spot $w$ must already be filled before $t$, or a car must park in spot $w$ at time $t$. Therefore any parking space filled under $Q$ at time $t$ must be filled under $G$ at time not later than $t$.

Now, by the inductive hypothesis, any car arriving at position $w$ under $Q$ at time $t$ must arrive at position $w$ under $G$ at time $t$. If $w$ is not a free parking space under $G$ at time $t-1$, then $W^w_Q(t) \subseteq W^w_G(t)$ and the claim holds. Thus suppose that $w$ is a free parking space at time $t-1$ under $G$. Then by the argument above, $w$ must be a free parking space at time $t-1$ under $Q$. Further, if under $G$ a car not from $W^w_Q(t)$ parks at $w$ at time $t$, then again $W^w_Q(t) \subseteq W^w_G(t)$ and again we are done. So suppose that under $G$ a car $v \in W^w_Q(t)$ parks at position $w$ at time $t$. By the tie-breaking procedure, $v$ must have the smallest $U^x_t$ value over the cars $x$ that arrive at $w$ under $G$, and so must have the smallest $U^x_t$ value over cars $x$ that arrive at $w$ under $Q$. Therefore under $Q$ the car $v$ must also park at $w$ at time $t$, and so once again we have $W^w_Q(t) \subseteq W^w_G(t)$.

Now, consider that the set of unparked cars at time $t$ is the union $\bigcup_{w \in V}W^w(t)$, hence if a car $v$ is still unparked under $Q$ at time $t$, then there is some $w \in L$ such that $v \in W^w_Q(t)$. But we know that $W^w_Q(t) \subseteq W^w_G(t)$, therefore we have $v \in W^w_G(t)$, implying that $\tau^v_Q \le \tau^v_G$ as desired.

Also, the number of visits to $w$ at time $t$ is
\[
|W^w(t)|+ \indi_{\{\mbox{\small a car parks at } w \mbox{ \small at time } t\}}.
\]
Since $W^w_Q(t) \subseteq W^w_G(t)$, and additionally we know that
\[
\sum_{s=1}^t \indi_{\{\mbox{\small a car parks under } Q \mbox{ at } w \mbox{ \small at time } s\}} \leq \sum_{s=1}^t \indi_{\{\mbox{\small a car parks under } G \mbox{ at } w \mbox{ \small at time } s\}},
\]
the inequality $V^v_Q(t) \le V^v_G(t)$ follows.
\end{proof}

\section{Probabilistic bounds}
\label{sec:probabilities}
In this section, we state some probabilistic bounds that are needed for the proofs in Sections \ref{sec:p=1/2upper}, \ref{sec:p=1/2lower}, and \ref{sec:p<1/2}. 

We make use of the following variant of the Chernoff bound (see \cite[Chapter~4]{Chernoffcite}).

\begin{lem}\label{feelthechern}
Let $p \in (0,1),$ $N \in \bN,$ and $\varepsilon > 0.$ Then $$\bP\left[\Bin(N,p) \ge N(p+\varepsilon)\right] \le e^{-2\varepsilon^2N}.$$
\end{lem}

We need some facts about hitting times of the simple symmetric random walk.
\begin{lem}\label{lbeforer}
Let $a, b > 0$ be positive integers. Let $\{X_n\}_{n \geq 0}$ be a simple symmetric random walk on $\bZ$ with $X_0 = 0.$ For $i \in \bZ,$ let $H_i = \min\{s : X_s = i\}.$ Then
	\begin{itemize}
		\item[(i)] $\bP[H_b < H_{-a}] = \frac{a}{a+b}.$ 
		\item[(ii)] $\bE[H_b | H_b < H_{-a}] = \frac{b(b+2a)}{3}.$
		\item[(iii)] $\bE[H_{-a} \wedge H_{a}] = a^2.$
	\end{itemize}
\end{lem}

\begin{proof}
All of this is standard. Part (i) is Gambler's ruin (see \cite[XIV.2]{Feller-Probability}). Part (iii) follows from (ii) by symmetry and a simple calculation.

For part (ii), we first prove the statement in a slightly different setup. Let $c,d$ be positive integers with $0 < c < d$ and assume that $X_0 = c.$ We show that
\[
\bE[H_d | H_d < H_0] = \frac{(d-c)(d+c)}{3}.
\]
Part (ii) of the lemma then follows immediately by taking $d = a+b$ and $c=a$. Let $Z_n = X_n^3 - 3nX_n$ and let $S =  X_{n+1}-X_n \in \{-1,1\}$. Then, since $S^2 = 1$ and $S^3 = S$, we have
\begin{align*}
Z_{n+1} & = \left(X_n+S\right)^3 - 3(n+1)\left(X_n+S\right) \\
 & = Z_n + 3X_n^2S+3X_n S^2+S^3 - 3nS - 3X_n -3 S \\
 & = Z_n + S\left(3X_n^2-2-3n\right).
\end{align*}

Since $X_{n+1}-X_n$ takes values in $\{-1,+1\}$ with mean $0$ independently of $\cF_n,$ and since $X_n$ is $\cF_n$-measurable, we have
\begin{align*}
		\bE[Z_{n+1} | \cF_n] & = \bE[Z_n + \left(X_{n+1}-X_n\right)\left(3X_n^2-2-3n\right) | \cF_n] \\
		& = Z_n + \left(3X_n^2-2-3n\right)\bE[X_{n+1}-X_n] = Z_n,
\end{align*}
and so $Z$ is a martingale.

For $n \in \bN,$ Doob's Optional Stopping Theorem gives $\bE[Z_{n\wedge H_0 \wedge H_d}] = \bE[Z_0] = c^3.$ At the same time, $|Z_{n\wedge H_0 \wedge H_d}|$ is bounded by $3d^3 + 3(H_0 \wedge H_d)d$ for all $n.$
Additionally, $H_0 \wedge H_d$ is integrable and so by the Dominated Convergence Theorem we have
\[
  \bE[Z_{H_0 \wedge H_d}] = \lim_{n \to \infty} \bE[Z_{n\wedge H_0 \wedge H_d}] = c^3.
\]
But $Z_{H_0 \wedge H_d} = \indi_{H_d < H_0}(d^3 -3dH_d).$ Therefore
	\begin{align*}
		c^3 = \bE[Z_{H_0 \wedge H_d}] & = \bE[\indi_{H_d < H_0}(d^3 -3dH_d)]  \\
		& = \bP\left[H_d < H_0\right](d^3 - 3d\bE[H_d | H_d < H_0]).
	\end{align*}
By (i), $\bP\left[H_d < H_0\right] = c/d,$ and so $\bE[H_d | H_d < H_0] = \frac{d^3 - c^2d}{3d} = \frac{d^2-c^2}{3}.$
\end{proof}

Let $M_n$ denote the maximum value in the first $n$ time steps of the simple symmetric random walk starting at $0,$ and let $m_n$ denote its corresponding minimum value. Define $p_{n,r} = \binom{n}{\frac{n+r}{2}} 2^{-n}.$ It can be shown (see, e.g., \cite[Theorem III.7.1]{Feller-Probability}) that for $r \geq 0$ we have
\[
 \bP\left[M_n = r\right] = \bP\left[m_n = -r\right] = \begin{cases}
p_{n,r} \quad & \mbox{if $n - r$ is even}, \\
p_{n,r+1} \quad & \mbox {otherwise.}
\end{cases}
\]

Let $Y \sim \Bin(n,1/2)$, where we will assume that $n$ is even, so that for $k \le n/2,$ $\bP\left[Y = \frac{n+2k}{2}\right] = p_{n,2k}.$ We now conclude this section with some tail bounds for the maximum of the random walk. We remark that the analogous results hold for $m_n$ by symmetry.

\begin{lem}\label{minssrw}
\begin{itemize}
\item[(i)] $\bP\left[M_n \ge 2 \alpha  \sqrt{n \log n}\right] \le 2 n^{-2 \alpha^2}.$
\item[(ii)] $\bP\left[M_n \ge c\sqrt{n}\right]$ and $\bP\left[M_n \le c\sqrt{n}\right]$ are bounded away from zero for each $c>0.$
\end{itemize}
\end{lem}

\begin{proof}
For (i) we have
	\begin{align}
		\bP\left[M_n \ge 2k\right] & = \bP\left[M_n = 2k\right] + \sum_{\ell = k+1}^{n/2} \left ( \bP\left[M_n = 2\ell-1\right] + \bP\left[M_n=2\ell \right] \right ) \nonumber \\
		& = p_{n,2k} + \sum_{\ell = k+1}^{n/2} \left ( p_{n,2\ell-1} + p_{n,2\ell} \right ) \nonumber \\
		& = p_{n,2k} + 2 \sum_{\ell = k+1}^{n/2} p_{n,2\ell} \nonumber \\
		& = \bP\left[Y = \frac{n+2k}{2}\right] + 2 \sum_{\ell=k+1}^{n/2} \bP\left[Y = \frac{n+2\ell}{2}\right] \label{double-ish} \\
		& \leq 2\bP\left[Y \ge \frac{n+2k}{2}\right]. \nonumber
	\end{align}

The same holds for odd $n,$ and so we see that $\bP\left[M_n \ge 2k\right] \le 2\bP\left[Y \ge n(1/2 + k/n)\right].$ Setting $k = \alpha \sqrt{n\log n}$ and applying Lemma \ref{feelthechern} gives
	\[
		\bP\left[M_n \ge 2 \alpha\sqrt{n \log n}\right] \le 2n^{-2 \alpha^2}.
	\]
	
For (ii), setting $k = c\sqrt{n}$ we get
	\begin{align*}
		\bP\left[M_n \le 2k\right] & \ge 1-\bP\left[M_n \ge 2k\right]  \\
		& \ge 1 - 2 \bP\left[Y \ge n(1/2 + k/n)\right]  \\
		& = 1 - 2\bP\left[\frac{Y-n/2}{\sqrt{n/4}} \ge 2c\right]  \\
		& \rightarrow 1 - 2(1-\Phi(2c)),   \\
		& = 2\Phi(2c)-1, 
	\end{align*}
as $n \rightarrow \infty$ by the Central Limit Theorem. Since $c > 0$ we have $\Phi(c) > 1/2,$ and so $\bP[M_n \le c\sqrt{n}]$ is bounded away from zero for each $c>0.$

From \eqref{double-ish}, we may read off $\bP\left[M_n \ge 2k\right] \ge 2 \bP\left[Y \ge \frac{n+2k+2}{2}\right] = 2 \bP\left[Y \ge \frac{n+2k}{2}\right]-O(n^{-1/2}),$ and so again for $k = c\sqrt{n},$ 
	\begin{align*}
		\bP\left[M_n \ge 2k\right] & \geq 2 \bP\left[Y \ge \frac{n+2k}{2}\right] -o(1) \\
		& = 2 \bP\left[\frac{Y-n/2}{\sqrt{n/4}} \ge 2c\right] -o(1) \\
		& \rightarrow 2-2\Phi(2c) > 0 
	\end{align*}
as $n \rightarrow \infty$ by the Central Limit Theorem. Therefore $\bP\left[M_n \ge c\sqrt{n}\right]$ is bounded away from zero for each $c>0$.
\end{proof}

We remark that Lemma \ref{minssrw} part (ii) also follows naturally from the convergence of the simple symmetric random walk to Brownian motion in the uniform topology.
\section{Upper bound on $\bE [\tau \wedge t]$}
\label{sec:p=1/2upper}

In this section, we prove the upper bound in Theorem \ref{thm:p=1/2}. We fix a target time $t$ and consider a particular weakly translation invariant parking strategy (specific to $t$) with additional properties. The parking strategy assigns (at time $0$) a parking space to most of the cars and tells the other cars they can never park. Each car then drives until it reaches its assigned parking place (or just keeps driving if it has no assigned space). The work left to do is to show that many cars are assigned parking spaces that they will reach in a short expected amount of time. We split this section into two parts; the first one detailing the parking strategy and showing some of its properties, and the second one bringing everything together to prove the desired upper bound.

\subsection{The parking strategy}
Fix $t \ge 1.$ We define the parking strategy $T = T_t$ as follows. We first divide $\bZ$ into intervals of length $\lceil \sqrt{t} \rceil.$ On each interval $I,$ we run through the locations from right to left, attempting to assign to each car $i$ a parking space $P(i)$ somewhere in $I$ and to the left of $i.$ If there is no unassigned parking space available within distance $O(t^{1/4})$ then car $i$ will not try to park, and we set $P(i) = \star;$ and if $i$ is a parking space we set $P(i) = i.$ This defines a strategy that is periodic, but not weakly translation invariant (because the intervals have specified endpoints). So we begin by applying a random shift to our intervals to make the strategy weakly translation invariant.

More formally, let $\zeta = \lceil \sqrt{t} \rceil$ and $\nu = \lceil t^{1/4} \rceil.$ First let $Z$ be uniformly distributed on $[\zeta]$ independently from the original model. Then, given $Z=z,$ for each interval $[z+k\zeta,z+(k+1)\zeta-1]$ we assign specific parking spaces to cars as follows:

\begin{algorithm}[H]
\SetKw{KwFn}{Initialization}
    \KwFn{Set $m=z+(k+1)\zeta-1,$ $W=\emptyset$}\;
    \While{$m \ge z+k\zeta$}{
    \eIf{\rm{There is initially a parking space at $m$}}{
    	    Set $P(m) = m$\;
    	    \If{$W \neq \emptyset$}{
            	Let $v$ be the largest element of $W.$\:
            	Remove $v$ from $W$ and set $P(v) = m$\;
            }
        }
      {There is initially a car at $m$. Add $m$ to $W$\;
        }
     \If{$|W| = \nu$}{
            Let $v$ be the largest element of $W.$\:
            Remove $v$ from $W$ and set $P(v) = \star$\;
        }
    Set $m := m-1$\;
    }
\SetKw{KwFm}{Finalization}
  \KwFm{For all $v \in W,$ set $P(v) = \star.$}
\label{discardalgorithmpark}
\end{algorithm}
\bigskip

The strategy $T$ is defined as follows: for each car $i,$
\begin{itemize}
\item if $P(i) = \star,$ then $S_t(i,j)=0$ for all $t \ge 1,$ $j \in \bZ$ (car $i$ never parks).
\item if $P(i) \neq \star,$ then $S_t(i,P(i)) = 1$ for the first time $t$ when car $i$ visits $P(i),$ and $S_t(i,j) = 0$ otherwise.
\end{itemize}

Note that the random variable $Z$ causes this parking strategy to be weakly translation invariant, and so it is sufficient to show that $\bE^T[\tau \wedge t] = O(t^{3/4})$ to prove the upper bound in Theorem \ref{thm:p=1/2}.

The benefit of this parking strategy is that it is much easier to give bounds on the expected hitting time of a fixed vertex rather than an arbitrary empty parking space. However, there are a couple of potential problems: the parking strategy might assign cars to distant parking spaces; and the parking strategy might dictate that many cars never park ($P(v) = \star$ for too many $v$). The next two lemmas resolve these problems.

\begin{lem}\label{sorting}
For all $i$ we have $P(i) = \star$ or $i-P(i) \le 2\nu-1.$
\end{lem}

\begin{lem}\label{markov}
For all $i \in \bZ,$ $\bP\left[P(i) = \star\right]  = O(t^{-1/4}).$
\end{lem}

Lemma \ref{sorting} follows from our choice to abandon the oldest car when the queue is too long.
\begin{proof}[Proof of Lemma \ref{sorting}]
Suppose that $i$ is a car and that after it joins the queue $W$ we have $|W|=q$. Further let $r$ be the number of cars assigned a parking space or removed from $W$ and having $P$ set as $\star$ in the next $2\nu-1$ loops. If $r \ge q$ then we are done since $i$ is $q$-th in the queue to either be assigned a parking space or removed and have $P(i)$ set as $\star$. Otherwise, if $r < q$, the queue is never emptied in the next $2\nu-1$ loops and we must see at most $q-1$ parking spaces which implies that we have at least $2\nu-q$ new cars added to the queue. Thus after those next $2\nu-1$ loops the number of cars in the queue must be at least $q + (2\nu-q)-(q-1) = 2\nu+1-q > \nu$. This is a contradiction since the queue can never be longer than $\nu$ cars.
\end{proof}

The proof of Lemma \ref{markov} is a little more involved. We use some elementary properties of irreducible, aperiodic Markov chains.

\begin{proof}[Proof of Lemma \ref{markov}]
We may assume without loss of generality that $Z$ is $0$, and we consider the interval obtained by taking $k=0$. By symmetry and translation invariance, we see that for any $i \in \bZ$
	\begin{equation}
		\bP^T\left[P(i) = \star\right] = \zeta^{-1}\bE^T[\left|\{j \in [0,\zeta-1] : P(j) = \star\}\right|]. \label{symsneak}
	\end{equation}

Let $C_n$ be the size of $W$ just before the last \textbf{if} clause of the loop when $m = \zeta-n,$ and set $C_0 = 0.$ In most situations we can only have $C_{n+1}-C_n$ equal to either 1 (if $\zeta-n-1$ is a car) or $-1$ (if $\zeta-n-1$ is a parking space). However, there are two exceptions to that rule. If $C_n = 0,$ i.e., if $W = \emptyset$ after we observe $\zeta-n,$ and if $\zeta-n-1$ is a parking space, then $C_{n+1} = 0$ as well. Moreover, if $C_n = \nu$ then in the last \textbf{if} clause of the loop we deterministically remove one element from $W.$ Thus depending on the value of whether $\zeta-n-1$ is a car o a parking space, we might have either $C_{n+1} = \nu$ or $C_{n+1} = \nu-2.$ Hence $C = (C_0, C_1, \ldots)$ is a Markov chain with transition probabilities $(p_{k,l})_{k,l \in \{0,\ldots,\nu\}}$ satisfying:
	\begin{itemize}
		\item $p_{0,0} = 1/2$ (there is a parking space but no queue),
		\item $p_{0,1} = 1/2$ (a car joins an empty queue),
		\item $p_{k,k-1} = 1/2$ when $i \in \{1,\ldots,\nu-1\}$ (a car in the queue is assigned a parking space),
		\item $p_{k,k+1} = 1/2$ when $i \in \{1,\ldots,\nu-1\}$ (a new car joins the queue),
		\item $p_{\nu,\nu-2} = 1/2$ (we tell an old car to leave the queue, and assign another queueing car to a parking space),
		\item $p_{\nu,\nu} = 1/2$ (we tell an old car to leave the queue, and a new car joins the queue),
		\item $p_{k,l} = 0$ otherwise.
	\end{itemize}
	
We see that some vertex gets assigned $\star$ each time $C$ hits $\nu.$ Additionally, the $C_{\zeta}$ vertices remaining in $W$ at the end of the execution of the algorithm also get assigned $P(v) = \star.$ Therefore
	\beqs
		\left|\{j \in [0,\zeta-1] : P(j) = \star\}\right| = C_{\zeta} + \sum_{n=0,\ldots,\zeta-1} \indi_{C_n = \nu} \label{compare}
	\eeqs
In our algorithm, we initially impose that $W = \emptyset.$ If, however, we started the algorithm with $W'$ containing some cars, then at every step in the algorithm, we would have $W' \supseteq W.$ Let $C'_n$ be the size of $W'$ just before the last \textbf{if} clause of the loop when $m = \zeta - n.$ Then we see that $\{C'_n\}$ is a Markov chain with transition probabilities $(p_{k,l})_{k,l \in \{0,\ldots,\zeta\}}$ such that $C'_n \ge C_n$ for all $n$. Thus, if $|W'|$ initially has distribution $\mu,$ we see
	\[
		\bP[C_n = \nu] \le \bP\left[C'_n = \nu\right] = \bP_{C_0 \sim \mu}\left[C_n = \nu\right].
	\]
In particular, if we let $\pi$ be a stationary distribution of $C,$ then for all $n$
	\[
		\bP_{C_0=0}\left[C_n = \nu\right] \le \bP_{C_0 \sim \pi}\left[C_n= \nu\right] = \pi(\nu).
	\]
Hence if we take the expectation of \eqref{compare} we obtain
	\[
		\bE^T[\left|\{j \in [0,\zeta-1] : P(j) = \star\}\right|] \le \bE^T[C_{\zeta}] + \zeta\pi(\nu).
	\]

Since $C$ is irreducible and aperiodic, and has a finite state space, it has a unique stationary distribution $\pi.$ One can then verify that $\pi(k) = \frac{1}{\nu}$ for $k=0,\ldots \nu-2,$ and $\pi(k) = \frac{1}{2\nu}$ for $k = \nu-1,\nu.$ Since $C$ takes values in $0,\ldots,\nu,$ we may bound $\bE[C_{\zeta}]$ by $\nu$ to find
	\[
		\bE^T[\left|\{j \in [0,\zeta-1] : P(j) = \star\}\right|] \le \nu + \frac{\zeta}{2\nu}.
	\]
Together with \eqref{symsneak} we obtain $\bP^T \left[P(i) = \star\right] \le \frac{\nu}{\zeta} + \frac{1}{2\nu}= O(t^{-1/4}).$
\end{proof}

\subsection{Proof of the upper bound}
We now have all the ingredients necessary to prove the upper bound in Theorem \ref{thm:p=1/2}. We will do this by bounding $\bE^T[\tau \wedge t]$ and then appealing to Corollary \ref{lem:supermarket}.
\begin{proof}[Proof of the upper bound in Theorem \ref{thm:p=1/2}]
Let $t \ge 0.$ Without loss of generality we can consider $\tau = \tau^0.$ Then
	\begin{align*}
		\bE^T[\tau \wedge t] &= \bE^T[\tau^0 \wedge t] \\
		&= \bE^T[\tau^0 \wedge t | P(0) = \star]\bP^T\left[P(0) = \star\right] + \bE^T[\tau^0 \wedge t | P(0) \neq \star]\bP^T\left[P(0) \neq \star\right] \\
		&\le t\bP^T\left[P(0) = \star\right] + \bE^T[\tau^0 \wedge t | P(0) \neq \star].
	\end{align*}
	
Lemma \ref{markov} gives $\bP^T[P(v) = \star] = O(t^{-1/4})$ and so
	\begin{equation}
      \label{uyp1}
		\bE^T[\tau^0 \wedge t] \le \bE^T[\tau^0 \wedge t | P(0) \neq \star] + O(t^{3/4}).
	\end{equation}

Let $a = 2\nu = 2 \lceil t^{1/4} \rceil, b = \zeta = \lceil \sqrt{t} \rceil$. For an integer $m,$ let $H_m$ be the first hitting time of the random walk $X^0$ to $m.$ Lemma \ref{sorting} tells us that if $P(0) \neq \star,$ then $P(0) \ge -a.$ We therefore see $\tau^0 \wedge t = H_{P(0)} \wedge t \le H_{-a}.$ When $H_{-a} > H_b,$ we may trivially bound $\tau^0 \wedge t$ by $t.$ Putting this into \eqref{uyp1} gives
	\begin{align*}
		\bE^T[\tau^0 \wedge t]  \le \bE^T[H_{-a} &| H_{-a} < H_b, P(0) \neq \star]\bP^T[H_{-a} < H_b | P(0) \neq \star] \\
		& + t\bP^T\left[H_{-a} > H_b | P(0) \neq \star\right] + O(t^{3/4}).
	\end{align*}
	
Clearly $X^0$ is independent from $P(0),$ which only depends on the initial configuration, and so
	\beq
		\bE^T[\tau^0 \wedge t] \le \bE[H_{-a} | H_{-a} < H_b] + t\bP\left[H_{-a} > H_b\right] + O(t^{3/4}).
	\eeq
	
Lemma \ref{lbeforer} (i) and (ii) tells us that $\bP\left[H_b<H_{-a}\right] = O(t^{-1/4})$ and $\bE[H_{-a} | H_{-a}<H_b] = O(t^{3/4}).$ We therefore see that
	\beq
		\bE^T[\tau^0 \wedge t] \le O(t^{3/4}) + tO(t^{-1/4}) + O(t^{3/4}) = O(t^{3/4}).
	\eeq
	
Finally, we appeal to Corollary \ref{lem:supermarket} to obtain
\[
\bE[\tau \wedge t] = \bE^G[\tau \wedge t] \leq \bE^T[\tau \wedge t] = O(t^{3/4}).
\]
\end{proof}

\section{Lower bound on $\bE^T[\tau \wedge t]$}
\label{sec:p=1/2lower}

In this section, we prove the lower bound in Theorem \ref{thm:p=1/2}. We do this by considering a parking process on an interval, and appealing to various properties of the simple symmetric random walk. While the underlying ideas are relatively simple, proving them rigorously requires a number of steps and some new ideas. We start with an outline of the proof.

\subsection{Outline of proof}
Instead of considering the expected journey length up to time $t,$ we consider the expected proportion of cars that have parked by time $t.$ It is helpful to restrict ourselves to a finite interval, and this is where car removal strategies become useful. We know by Theorem~\ref{thm:pinkfloyd} that by removing cars from the process we make it easier for the remaining cars to park. Therefore, any lower bound over an interval for the proportion of unparked cars gives a lower bound for $\bE[\tau \wedge t].$

From here, we consider a long interval $L \cup M \cup R,$ where $L,M,R$ are the left, middle, and right subintervals respectively. We will choose the sizes of $L$ and $R$ so that with high probability no car from $M$ leaves $L \cup M \cup R$ by time $t.$ The idea is that with positive probability the number of cars starting in $M$ is a few standard deviations above the mean, creating an excess of cars, and that this excess is not relieved by what happens in $L$ and $R.$ To be able to quantify this, we introduce \em swapping: \em this is a way of switching positions of cars so that at any time, from left to right, we see the cars that started in $L,$ then $M,$ and then $R.$ This modification does not change the stochastic properties of the process, but does allow us to say how much relief $L$ and $R$ provide by way of parking spaces available to cars starting in $M.$

Finally, we will bring everything together and appeal to Theorem \ref{thm:pinkfloyd} to obtain the desired lower bound on $\bE[\tau \wedge t].$

\subsection{The car removal strategy and the swap-modification}
We define the car removal strategy $Q$ as follows. Fix integers $k >8$ and $\ell >4.$ Let $\zeta = \lceil \sqrt{t\log t} \rceil$. Then for each integer $r \in \bZ$ we remove any car which attempts to make a step (in either direction) between $r(2(k+\ell)\zeta+1)$ and $r(2(k+\ell)\zeta+1)+1$.

We show that a proportion $(t\log t)^{-1/4}$ of cars remains active (i.e., unparked and not removed) at time $t$ under the car removal strategy $Q.$ To establish this, it is sufficient to consider the parking process on an interval of length $2(k+\ell)\zeta+1$ where we assume that cars leaving the interval at either end are removed. Let $L = \bZ \cap [-(k+\ell)\zeta,-k\zeta),$ let $M = \bZ \cap [-k\zeta,k\zeta]$ and $R = \bZ \cap (k\zeta,(k+\ell)\zeta].$

We want to show that with positive probability we start with an excess of cars in $M$ which do not escape $L \cup M \cup R$ and that $L$ and $R$ do not offer up enough spare parking capacity. It turns out that quantifying what capacity $R$ and $L$ provide is not straightforward since one cannot easily separate what happens to the cars with respect to their starting positions. Particularly problematic is that cars starting in different sections ($L, M$, or $R$) may swap positions. The following modification of the process ensures that at any given time the active cars, as seen from left to right, started their journeys in $L$, then in $M$, and finally in $R$, and will prove very useful.

\begin{defn}[The modified parking process]\label{con:modify}
Given the parking process $X$, we define a modified process $Y$ as follows. At time $0,$ label cars according to their starting intervals $L,M$ or $R.$ For $s\geq 0$ we write $C(s)$ for the set of starting positions (in $L \cup M \cup R$) of the cars that are still active at time $s$ (hence $C(0)$ is the set of $i$ such that we initially place an active car at $i$). Further we write $C_L(s)$ for the set of starting positions of the cars that started in $L$ and are still active at time $s;$ we similarly define $C_M(s)$ and $C_R(s)$. For a car starting at $i$ which is still active at time $s$ we write $Y^i(s)$ to denote its position at time $s.$

Given the set $C(s)$ of cars active at time $s,$ and their positions $(Y^i(s): i \in C(s)),$ we want to define $C(s+1)$ and the positions $(Y^i(s+1) : i \in C(s+1)).$ We do this in several steps: at each step, we move the cars around in a way that preserves the number of cars at each location. We use $Z^i_1,Z_2^i,$ and $Z_3^i$ to denote intermediate rearrangements, preserving $Y^i$ for the final position. 

Roughly speaking: $Z_1$ is where the cars move according to their respective random walks. From $Z_1$ to $Z_2$ we swap cars so that no $L$-car is to the right of an $R$-car. From $Z_2$ to $Z_3$ we swap cars so that no $L$-car is to the right of an $M$-car. Finally, from $Z_3$ to $Y$ we swap cars so that no $M$-car is to the right of an $R$-car. The end result is a swapping of cars which preserves the number of cars at each vertex, is such that cars move by at most one in a single time step, and is such that from left to right the cars have labels $L$, then $M$, and then $R$.

	\begin{itemize}
		\item For any car active at time $s,$ define $Z^i_1(s+1) = Y^i(s) + (X^i(s+1)-X^i(s)).$
		\item Let $i_1, \ldots, i_{x} \in L$ (with $Z_1^{i_k}(s+1)$ increasing in $k$) be the starting positions of cars labelled $L$ that are active at time $s$ and such that the move at time $s+1$ places them to the right of some active car labelled $R.$ Similarly, let $j_1,\ldots,j_{y} \in R$ (with $Z_1^{j_k}(s+1)$ increasing in $k$) be the starting positions of cars labelled $R$ that are active at time $s$ and such that the move at time $s+1$ places them to the left of some active car labelled $L.$ 

		We rearrange the cars as follows: for all $i \notin \{i_1, \ldots, i_{x}, j_1,\ldots,j_{y} \}$ let $Z^i_2(s+1) = Z^i_1(s+1).$ Let $(m_1, \ldots, m_{x+y})$ be a permutation of $\{i_1, \ldots, i_{x}, j_1, \ldots, j_{y}\}$ with $Z^{m_k}_1(s+1)$ increasing in $k$. Then, for $1 \leq \ell \leq x,$ let $Z^{i_\ell}_2(s+1) = Z^{m_\ell}_1(s+1),$ and for $1 \leq \ell \leq y,$ let $Z^{j_\ell}_2(s+1) = Z^{m_{x+\ell}}_1(s+1).$ After this procedure, no car labelled $L$ is to the right of a car labelled $R.$		
		\item Given $Z^i_2(s+1)$ for all $i \in C(s),$ we define $Z^i_3(s+1)$ by reordering in a similar way the positions $Z^i_2(s+1)$ of the cars that started in $L$ or in $M$ in such a way that no car that started in $L$ has a car that started in $M$ to its left:

Let $i_1, \ldots, i_{x} \in L$ (with $Z_2^{i_k}(s+1)$ increasing in $k$) be the starting positions of cars labelled $L$ that are active at time $s$ and such that the move at time $s+1$ and the previous rerrangement places them to the right of some active car labelled $M.$ Similarly, let $j_1,\ldots,j_{y} \in M$ (with $Z_2^{j_k}(s+1)$ increasing in $k$) be the starting positions of cars labelled $M$ that are active at time $s$ and such that the move at time $s+1$ and previous rearrangement places them to the left of some active car labelled $L.$ 

		We rearrange the cars as follows: for all $i \notin \{i_1, \ldots, i_{x}, j_1,\ldots,j_{y} \}$ let $Z^i_3(s+1) = Z^i_2(s+1).$ Let $(m_1, \ldots, m_{x+y})$ be a permutation of $\{i_1, \ldots, i_{x}, j_1, \ldots, j_{y}\}$ with $Z^{m_k}_2(s+1)$ increasing in $k$. Then, for $1 \leq \ell \leq x,$ let $Z^{i_\ell}_3(s+1) = Z^{m_\ell}_2(s+1),$ and for $1 \leq \ell \leq y,$ let $Z^{j_\ell}_3(s+1) = Z^{m_{x+\ell}}_2(s+1).$

Note that this operation can only move cars labelled $L$ to the left; hence we still have no car labelled $L$ to the right of a car labelled $R.$
		\item Finally, given $Z^i_3(s+1)$ for all $i \in C(s),$ we define $Y^i(s+1)$ by reordering in a similar way the positions $Z^i_3(s+1)$ of the cars that started in $M$ or in $R$ in such a way that no car that started in $R$ has a car that started in $M$ to its right. Again, note that this operation only moves  cars labelled $R$ to the right, hence we still have no car labelled $L$ to the right of a car labelled $R.$ Moreover, a car labelled $M$ can only be moved to a position $Z^i_3$ previously occupied by a car labelled $M$ or $R,$ which we know has no car labelled $L$ to its right; hence the same holds about cars labelled $M$ after the rearrangement.
	\end{itemize}

If a single car starting at $i$ reaches an empty parking space at $Y^i(t),$ then it parks there. When at least two cars simultaneously arrive at a parking space $v$ at time $t,$ we choose the car $i$ labelled $L$ with smallest $U^i_t$ to park there; in the absence of a car labelled $L,$ the car $i$ labelled $R$ with smallest $U^i_t$ parks there; finally, if only cars labelled $M$ meet at $v,$ the car $i$ with smallest $U^i_t$ parks there. When a car leaves $L \cup M \cup R,$ we say it is \em inactive \em and remove it from the process. We say that a car becomes \emph{left-inactive} if it reaches $\min L -1,$ and it becomes \emph{right-inactive} if it reaches $\max R +1.$ Finally, let $C(s+1) \subseteq C(s)$ be the set of cars active at time $s$ that have neither parked nor become inactive at time $s+1.$
\end{defn}

\begin{rem}
 \label{rem:limitedDrive}
In the process described in Definition \ref{con:modify}, for any $i \in \bZ,$ $Y^i(s+1) - Y^i(s) \in \{-1,0,+1\}$ -- the total move of a car in a step is at most one. Indeed, consider an arbitrary car $i$ labelled $M$ with $Y^i(s) = j$. At time $s$ it has no cars labelled $L$ strictly to its right and no cars labelled $R$ strictly to its left. At time $s+1,$ all cars labelled $L$ can only drive to positions at most $j+1$ (so $Z^k_1(s+1) \le j+1$ for each $k \in C_L(s)$), and cars labelled $R$ drive to positions at least $j-1$ (so $Z^k_1(s+1) \ge j-1$ for each $k \in C_R(s)$). It is not possible to move $i$ to a position strictly to the left of the left-most (according to $Z_1$) car labelled $R$ so that $Y^i(s) \ge j-1$. Similarly it is not possible to move $i$ to a position strictly to the right of the right-most (according to $Z_1$) car labelled $L$ so that $Y^i(s) \leq j+1$. Similar arguments apply to cars labelled $L$ or $R.$
\end{rem}

Let $\widetilde\bP$ be the probability measure with respect to the modified parking process. If we ignore the labels of the cars, then the difference from the original parking process under $Q$ is that we swap some future trajectories of cars. Since the swapping is determined by past trajectories, the unlabelled modified process has the same distribution as the original parking process with car removal strategy $Q$. Thus 
	\begin{align}
		\widetilde\bE[\#\mbox{active cars in $L \cup M \cup R$ at time $t$}] = \bE^Q[\#\mbox{active cars in $L \cup M \cup R$ at time $t$}]. \label{telematch}
	\end{align}

\subsection{Proof of the lower bound}
Before completing the proof of Theorem \ref{thm:p=1/2}, we prove some preliminary lemmas concerning the modified parking process. Unless stated otherwise, we assume that we are dealing with the modified parking process (Definition \ref{con:modify}) throughout this section.

First we consider how many cars from $L$ and $R$ become inactive. Intuitively this should be maximised if the cars drive monotonically towards the ends of the interval. Given the initial arrangement of cars and parking spaces on $L,$ let $D_L = D_L(t)$ be the number of cars starting in $L$ which would become left-inactive by time $t$ should all cars with label $L$ move left deterministically. Similarly let $D_R = D_R(t)$ denote the number of cars with label $R$ that become right-inactive by time $t$ in the process where all cars with label $R$ move right deterministically. The next lemma shows that this intuition is correct.

\begin{lem}\label{claim:lemming}
The number of cars with label $L$ which become left-inactive by time $t$ is at most $D_L.$
\end{lem}

We prove this by considering parking spaces $v$ left unfilled or filled by a car with label $M$ or $R$. Any car with label $L$ starting to the right of $v$ cannot venture to the left of such $v$ as it would have parked there. This restricts the cars with label $L$ which become left-inactive.

\begin{proof}
Under $\widetilde\bP,$ suppose that $j$ is the smallest integer which has a parking space either unfilled or filled by a car labelled $M$ or $R$ by time $t.$ Let $J = L \cap [-(k+\ell)\zeta,j-1].$ We claim that the only cars labelled $L$ that can become left-inactive are the cars from $J,$ and only cars originating in $J$ park in $J.$ First suppose that $j$ is unfilled. No car from the right of $j$ passes through $j$ (else it would park there) and so no car from the right of $j$ can become left-inactive.

So suppose that car $w$ (labelled $M$ or $R$) parks in $j$. Under $\widetilde\bP$ at any time, from left to right, the unparked cars have labels $L,$ then $M,$ and then $R.$ Therefore, any car $v$ labelled $L$ originating from an integer greater than $j,$ before it parks, must stay to the left of the car $w$ which parks in $j.$ Since cars in the modified process move at most one step at each time, the car $v$ cannot be unparked at time $t$ since it would have visited $j$ before $w$ parks there. Similarly, $v$ cannot park to the left of $j$ since it would first pass through $j$ (before $w$ parks there). Therefore, any car labelled $L$ originating from an integer greater than $j$ must have parked in a spot greater than $j.$

Suppose that cars starting at positions $i_1<\dots<i_N < j$ become left-inactive starting from $J.$ Observe that every parking space to the left of $i_N$ must be filled by a car originating from $J$ (otherwise, the car starting in $i_N$ must reach a free parking space on its route to $\min L -1$). Let $p=j-1$ if all parking places in $J$ are filled in the process, and otherwise let $p+1$ be the leftmost empty parking space in $J$ at time $t.$ We see that all parking spaces to the left of $p+1$ must be filled by cars originating from the left of $p+1$ (a car starting to the right of $p$ would fill $p+1$ first). But then there must be a surplus of $N$ cars to the left of $p+1.$

If all the cars drove left deterministically, this surplus would result in at least $N$ cars, starting to the left of $p+1,$ becoming left-inactive. Thus we have $D_L \geq N,$ proving the claim.
\end{proof}

\begin{rem}
Clearly, the analogous claim that the number of cars from $R$ becoming right-inactive is bounded by $D_R$, also holds.
\end{rem}

Note that $D_L$ and $D_R$ are dependent only on the initial car configuration $(B^i)_{i \in \bZ}.$ Let $S_L$ be the number of cars which start in $L$ and let $P_L$ be the number of parking spaces in $L$ (hence clearly $S_L+P_L = |L|$). Similarly define $S_R$ and $P_R.$

\begin{lem}\label{claim:easing}
There exists $\varepsilon > 0$ (independent of $t$) such that
\[
 \bP\left[S_L-P_L-D_L \ge -(t\log t)^{1/4}\right] > \varepsilon.
\]
\end{lem}

For this lemma we consider the random walk defined by the number of cars minus the number of parking spaces we see in the initial configuration in $L$ while going from right to left through the subinterval, and the relation between the minimum value of this random walk and the process where all cars in $L$ deterministically drive left.

\begin{proof}
Consider the simple symmetric random walk starting at $0$ which increases at time $i \geq 1$ if the $i$th rightmost point in $L$ initially contains a car, and decreases if the $i$th rightmost point in $L$ contains a parking space. Suppose that while traversing $L,$ the walk last attains its minimum value $-m \leq 0$ at time $j,$ and let $x$ be the $j$th rightmost point in $L.$ Then, in the process where all cars in $L$ deterministically drive left, every car starting to the right of $x$ finds a parking place, the process ends with $m$ empty spots to the right of $x-1,$ every spot to the left of $x$ is filled by a car, and all the cars that do not park reach the left end of the interval and become left-inactive.

The number of parked cars in this process is $S_L-D_L,$ and so the number of unfilled parking spaces is $P_L-S_L+D_L.$ Therefore $S_L - P_L-D_L = -m.$ From the previous paragraph, we see that $S_L - P_L-D_L$ is distributed like the minimum of a simple symmetric random walk of length $\ell \zeta.$ So by Lemma \ref{minssrw}(ii) it is at least $-(t\log t)^{1/4}$ with probability bounded away from zero.
\end{proof}
\begin{rem}
 An analogous claim holds if we replace $S_L, P_L, D_L$ with $S_R, P_R, D_R$ respectively.
\end{rem}

We would like to say that no car from $L$ becomes right-inactive. Indeed, we could then say that at time $t,$ the number of cars from $L$ (possibly parked) still in $L\cup M \cup R$ minus the number of parking spaces (filled or unfilled) in $L$ is at least $S_L-P_L-D_L \ge -(t\log t)^{1/4}$ with probability at least $\varepsilon.$ The next result shows that this occurs, and also that no car from $M$ becomes inactive.

\begin{lem}\label{claim:restrict}
With probability $1-o(1/t),$ the random walks $\left(X^i\right)_{i\in L \cup M \cup R}$ are such that for all possible starting configurations of active cars and parking places in $L \cup M \cup R,$ in the first $t$ time steps the following holds: no car starting in $M$ becomes inactive, no car starting in $L$ reaches $R,$ and no car starting in $R$ reaches $L.$
\end{lem}

To prove this Lemma we combine the results concerning the maximum value of a simple symmetric random walk given in Section \ref{sec:probabilities} and the effects of the swaps. Roughly speaking we note that the swaps respectively push cars from $L, M$ and $R$ to the left, middle, and right.

\begin{proof}
For each $i \in L \cup M \cup R,$ let $M^i$ be the maximum of $\{X^i_s - i : s \le t\},$ and $m^i$ the minimum of $\{X^i_s -i : s \le t\}.$ By Lemma \ref{minssrw} (i), $\bP\left[m^i \le - 4\sqrt{t \log t}\right]=\bP\left[M^i \ge 4\sqrt{t \log t}\right] \le 2t^{-8}.$ Hence by the union bound, with failure probability $o(t^{-1}),$ for all $i \in L \cup M \cup R$ the random walks $X^i$ are at distance at most $4\zeta$ from their corresponding starting point $i$ until time $t.$

Assume that for all $i \in L \cup M \cup R,$ $X^i$ is at distance at most $4\zeta$ from $i$ until time $t.$ We now show that for all starting configurations of active cars and parking places in $L \cup M \cup R,$ in the first $t$ time steps, no car starting in $M$ becomes inactive, no car starting in $L$ reaches $R,$ and no car starting in $R$ reaches $L.$

Consider a car starting at $i \in L.$ If the car is still active at time $s$ in the modified parking process, then $Y^i(s) \leq X^i(s),$ as if the position of the car is ever changed as a result of landing to the right of a car labelled $M$ or $R,$ then it can only be pushed further left. Therefore it stays to the left of $(4-k)\zeta.$ Similarly all cars labelled $R$ stay to the right of $(k-4)\zeta.$ Since $k > 8,$ no car from $L$ reaches $R,$ and vice versa.

Now consider a car starting at $i \in M.$ If the position of the car is never changed due to moving past a car labelled $L$ or $R,$ then it never reaches a point more than $4\zeta$ from $i$ and so cannot become inactive (recall that $\ell > 4$).

Hence suppose the car at some point has its position changed due to finding itself to the left of a car labelled $L.$ This implies that the car must at some point be to the left of $(4-k)\zeta$ (or else it cannot pass a car labelled $L$). If the car reaches $(k-4)\zeta + 1$ at some point, then there must be a passage of the car between $(4-k)\zeta$ and $(k-4)\zeta$ contained within $[(4-k)\zeta,(k-4)\zeta].$ In this segment, the position of the car cannot be changed as it keeps all cars labelled $L$ to its left, and all cars labelled $R$ to its right. Therefore it moves according to $X^i,$ and so $X^i$ reaches points $2(k-4)\zeta > 8 \zeta$ apart (recall that $k > 8$). This cannot happen since the maximum modulus of $X^i-i$ is at most $4\zeta.$

Therefore the car does not have its position changed due to being to the right of a car labelled $R.$ So while the car remains active, its position is bounded below by $X^i$ (having its position changed can only push its the car to the right). Since the car does not reach $(k-4)\zeta,$ we see that the position of the car is contained in $[(-k-4)\zeta,(k-4)\zeta]$ and so the car cannot become inactive (as $\ell > 4$).

The argument for a car which at some point finds itself to the right of a car labelled $R$ is identical. We conclude that no car originating from $M$ becomes inactive.
\end{proof}

We are now in a position to prove Theorem \ref{thm:p=1/2}.

\begin{proof}[Proof of the lower bound in Theorem \ref{thm:p=1/2}]
It is enough to show that with probability bounded away from zero (say at least $\delta > 0$), at time $t$ there are at least $(t \log t)^{1/4}$ active cars in $L \cup M \cup R$ in the modified process. If this holds, then the result easily follows by symmetry, Theorem \ref{thm:pinkfloyd} and \eqref{telematch}:
	\begin{align*}
		\bE[\tau \wedge t] & = \frac{\bE^N\left[\sum_{v \in L\cup M \cup R} \tau^v \wedge t\right]}{|L \cup M \cup R|} \\
		& \ge \frac{\bE^Q\left[\sum_{v \in L\cup M \cup R} \tau^v \wedge t\right]}{|L \cup M \cup R|} \\
		& \ge \frac{\widetilde\bE[\# \mbox{active cars at time $t$ in $L \cup M \cup R$}]}{2(k+\ell)\zeta+1} \cdot t  \\
		& = \frac{\widetilde\bE[\# \mbox{active cars at time $t$ in $L \cup M \cup R$}]}{2(k+\ell)\zeta+1} \cdot t  \\
		& \ge \frac{\delta (t \log t)^{1/4}}{2(k+\ell)\zeta+1} \cdot t  \\
		& = \Omega(t^{3/4} \log^{-1/4} t).
	\end{align*}

Let $I_L$ be the number of cars starting in $L$ that become left-inactive and let $I_R$ be the number of cars starting in $R$ that become right-inactive. Analogously to $L$ and $R$, let $S_M$ be the number of cars which start in $M$ and let $P_M$ be the number of initial parking places in $M.$ Hence, in total there are $P_L+P_M+P_R$ parking places in $L \cup M \cup R.$

Suppose that in the first $t$ steps of the process, no car starting in $M$ becomes inactive, no car starting in $L$ reaches $R,$ and no car starting in $R$ reaches $L.$ Then at time $t,$ the number of cars (active or parked) in $L \cup M \cup R$ is $S_M  + (S_L - I_L) + (S_R  - I_R).$ By Lemma \ref{claim:lemming} this is at least $S_M  + (S_L - D_L) + (S_R  - D_R).$ Since only one car can park in a parking space, the number of active cars in $L \cup M \cup R$ at time $n$ must be at least
	\beqs
		(S_M-P_M) + (S_L-P_L)-D_L + (S_R-P_R)-D_R. \label{eq:nexcess}
	\eeqs

Observe that $S_M-P_M$ is determined by the starting configuration in $M,$ $S_L-P_L-D_L$ is determined by the starting configuration in $L,$ and $S_R-P_R-D_R$ is determined by the starting configuration in $R.$ Therefore these random variables are mutually independent. Let $C^M$ be the event that $S_M - P_M$ is at least $3(t\log t)^{1/4},$ let $C^L$ be the event that $S_L-P_L-D_L \geq -(t\log t)^{1/4},$ and let $C^R$ be the event that $S_R-P_R-D_R \geq - (t\log t)^{1/4}.$

Let $A$ be the random event, depending on the random walks $X^i$ only, that for all initial configurations of cars and parking places in $L \cup M \cup R,$ no car from $M$ becomes inactive, no car from $L$ reaches $R,$ and no car from $R$ reaches $L.$ Observe that $A,C^L, C^M,C^R$ are mutually independent events. By Lemma \ref{claim:restrict}, $A$ occurs with high probability. By Lemma \ref{claim:easing} both $C^L$ and $C^R$ occur with probability bounded away from zero. Let $K = 2k\zeta + 1 \approx 2k\sqrt{t\log t}.$ Since $$S_M - P_M = 2k\zeta +1 - 2P_M \sim K - 2\mathrm{Bin}(K,1/2),$$ we have
	\begin{align*}
		\bP\left[C^M\right] & = \bP\left[\mathrm{Bin}(K,1/2) \leq \frac{K - 3(t\log t)^{1/4}}{2}\right]  \\
		& = \bP\left[\frac{\mathrm{Bin}(K,1/2) - \frac{K}{2}}{\sqrt{\frac{K}{4}}} \leq \frac{-6(t\log t)^{1/4}}{\sqrt{K}}\right].
	\end{align*}
By the Central Limit Theorem, this probability tends to $\Phi(-\frac{6}{\sqrt{2k}})$ as $t$ tends to infinity. Therefore $C^M$ occurs with probability bounded away from zero. So all four events $A,C^L, C^M,C^R$ occur simultaneously with probability bounded away from zero.

Suppose that the events $A,C^L,C^M,C^R$ all occur. Then recalling equation \eqref{eq:nexcess} we see that the number of active cars in $L \cup M \cup R$ at time $t$ is at least
\begin{align*}
 (S_M-P_M) + (S_R-P_R)-D_R + (S_L-P_L)-D_L & \ge 3 (t\log t)^{1/4} - (t\log t)^{1/4} - (t\log t)^{1/4} \\
  & = (t\log t)^{1/4}.
\end{align*}

We conclude that with probability bounded away from zero, there are at least $(t\log t)^{1/4}$ active cars in $L \cup M \cup R$ at time $t.$ The lower bound $\Omega(t^{3/4} \log^{-1/4} t)$ on $\bE[\tau \wedge t]$ follows.
\end{proof}

\section{Subcritical parking on $\bZ$}
\label{sec:p<1/2}

In this section we prove Theorem \ref{thm:p<1/2}. This is done in two parts. First, for a car starting at $0$, we consider the smallest $J$ (depending only on the initial configuration of cars) such that no matter what the other cars do, there is always a free parking space in both $[1,J]$ and $[-J,-1]$. Given $J$, we know that the car starting at $0$ cannot reach either $-J$ or $J$ before it parks. Calculating the expected journey length of $0$ is then carried out by proving tail bounds on the random variable $J$.

We start with the following simple lemma, which we state here without proof.

\begin{lem}\label{lem:markstat}
Let $p \in (0,1/2),$ and let $Y=Y(p)$ be a Markov chain on $\bN \cup \{0\}$ with $Y_0=0$ and transition probabilities $(p_{i,j})_{i,j \in \bN \cup \{0\}}$ where
\[
p_{i,j} = \begin{cases}
p, \quad & j = i+1, \\
1-p, \quad & j = i - 1 \ge 0 \mbox{ or } i=j=0, \\
0, \quad & \mbox {otherwise.}
\end{cases}
\]
Then $Y$ has stationary distribution $\Geom_{\ge 0}(\frac{1-2p}{1-p})$. Furthermore, since $Y$ is an aperiodic and irreducible Markov chain, $Y_t \rightarrow \Geom_{\ge 0}(\frac{1-2p}{1-p})$ in distribution as $t \rightarrow \infty$.
\end{lem}

Let $E^L(t)$ be the number of cars in $[-t,-1]$ that would reach $0$ if all cars deterministically drove right. We also define $E^R(t)$ to be the number of cars in $[1,t]$ which would reach $0$ if all cars deterministically drove left. Note that $(E^L(t))_{t \in \bN}$ is an increasing sequence of random variables. Finally, let $E^L$ be the number of cars in $(-\infty,-1]$ that would reach $0$ if all cars deterministically drove right (and analogously define $E^R$). Note that $E^L(t)$ increases almost surely to $E^L$ as $t \rightarrow \infty$.

\begin{lem}
\label{lem:thedescent}
For all $p < 1/2,$ $E^L \sim \Geom_{\ge 0}(\tfrac{1-2p}{1-p})$.
\end{lem}

We prove Lemma \ref{lem:thedescent} by comparing each $E^L(t)$ to the $t$-th step of a Markov chain $Q$ with transition probabilities $p_{i,j}$ defined as in the statement of Lemma \ref{lem:markstat}.

\begin{proof}
Since $(E^L(t))_{t \ge 1}$ increases almost surely to $E^L$, it is sufficient to show that $E^L(t) \to \Geom_{\ge 0}(\tfrac{1-2p}{1-p})$ in distribution as $t \to \infty$. To compute $E^L(t)$, consider forming a queue of cars from left to right in $[-t,-1]$: Let $Q_0 = 0$ (there is initially no queue), then given $Q_i$, we set $Q_{i+1} = Q_i+1$ if there is initially a car at $i-t$ (a car is added to the queue), $Q_{i+1} = Q_i-1$ if $Q_i > 0$ and there is initially a parking space at position $i-t$ (a car from the queue is parked), and $Q_{i+1}=0$ otherwise. Then $Q_t = E^L(t)$. On the other hand, $(Q_s : s \le t)$ is distributed like $(Y_s : s \le t)$ in Lemma \ref{lem:markstat}, and so $E^L(t)$ has the same distribution as $Y_t$ (with $Y_0 = 0$). By Lemma~\ref{lem:markstat}, $E^L(t) \rightarrow \Geom_{\ge 0}(\frac{1-2p}{1-p})$ in distribution as $t \rightarrow \infty$.
\end{proof}

Clearly, $E^R(t)$ also increases almost surely to the random variable $E^R$ which is distributed like a $\Geom_{\ge 0}(\frac{1-2p}{1-p})$ random variable (and is independent of $E^L$).

For all $r \ge 0$, let ${E}_r^R(t)$ be the number of cars in $[r+1,r+t]$ that would reach $r$ if all cars deterministically drove left, and similarly let ${E}_r^L(t)$ be the number of cars in $[-r-t,-r-1]$ that would reach $-r$ if all cars deterministically drove right. Let ${E}_r^R$ and ${E}_r^L$ be the limits as $t \to \infty$ respectively of ${E}_r^R(t)$ and ${E}_r^L(t)$. Note that Lemma \ref{lem:thedescent} holds with $E^L$ replaced by $E^R_r$, as well as by $E^L_r$. For all $r \ge 1$, let $S_r^R$ and $S_r^L$ be the number of cars that start in $[1,r]$ and $[-r,-1]$ respectively.

In the proof of Theorem \ref{thm:p<1/2}, we show that at most ${E}_K^R+E^L + S_K^R$ cars from $\bZ\setminus \{0\}$ can be present in $[1,K]$ at any time. This means that at most ${E}_K^R+E^L + S_K^R$ parking spaces in $[1,K]$ can be filled by cars from $\bZ \setminus \{0\}$. Therefore, if ${E}_K^R+E^L + S_K^R < K/2,$ then there must be a parking space in $[1,K]$ not filled by a car from $\bZ \setminus \{0\}$ (consider that there are initially $K-S_K^R$ parking spaces in $[1,K]$). It follows that a car starting at $0$ parks before reaching $K$.

In the proof of Theorem \ref{thm:p<1/2}, we first condition on the smallest $K$ such that both ${E}_K^R+E^L + S_K^R < K/2$ and ${E}_{K}^L+E^R + S_K^L < K/2.$ These conditions mean that a car starting at $0$ will have parked by the time its associated random walk $X^i$ hits either $-K$ or $K.$

\begin{proof}[Proof of Theorem \ref{thm:p<1/2}]
Let $p < 1/2$ and let $J$ be the smallest $K$ such that ${E}_K^R+E^L + S_K^R < K/2$ and ${E}_{K}^L+E^R + S_K^L < K/2$ if such a $K$ exists, and let $J = \infty$ otherwise. For a given $a \in \bN,$ let $H_a$ be the first hitting time of $a$ by the random walk $X^0.$ We claim that if $J=N,$ then $\tau^0 \le H_{-N} \wedge H_N.$ We justify this by showing that at any time $t \ge 0,$ there are at most ${E}_N^R+E^L + S_N^R$ cars excluding car $0$ (parked or not) present in $[1,N]$ at time $t.$ A similar statement can be shown for cars present in $[-N,-1].$

Let us temporarily exclude the car starting at $0$ from the parking process (e.g., assume that this car never decides to park) and suppose that at time $t,$ there are $B$ cars that started in $[N-t+1,N]\setminus \{0\}$ parked in $[N+1,N+t].$ Let $R$ be the number of cars that start in $[N+1,N+t]$ that are in $[N-t+1,N]$ at time $t.$ By an argument identical to that of Lemma \ref{claim:lemming}, we have the bound $R \le B + E_N^R(t)$ since each parked car from $[N-t+1,N]$ that parks inside $[N+1,N+t]$ can only increase the number of cars that reach $N$ from $[N+1,N+t]$ by $1.$ Similarly, if $C$ is the number of cars that started in $[1,t]$ and parked in $[-t,-1],$ and $L$ is the number of cars that start in $[-t,-1]$ present in $[1,N]$ at time $t,$ we have $L \le C + E^L(t).$ So the number of cars present in $[1,N]$ at time $t$ is
	\begin{align*}
		S_N^R + R + L - B - C & \leq S_N^R + E^L(t) + E_N^R(t) \\
		& \leq S_N^R + E^L + E^R_N.
	\end{align*}
Since $J = N,$ this quantity is strictly less than $N/2.$ On the other hand, there are initially $N-S_N^R > N/2$
parking spaces in $[1,N]$ and so car $0$ must go through an empty parking space before reaching $N.$ A similar argument applies to $[-N,-1].$ In the real process, where car $0$ tries to park, this implies car $0$ parks before reaching $N$ or $-N.$

If $J < \infty$ almost surely, we therefore have
	\[
		\bE[\tau^0] \le \sum_{N \ge 1}\bP\left[J=N\right]\bE[H_{-N} \wedge H_N | J = N].
	\]
By independence and Lemma \ref{lbeforer} (iii) we have
\[
	\bE[H_{-N} \wedge H_N | J = N] = \bE[H_{-N} \wedge H_N] = N^2,
\]
and so, assuming again that $J < \infty$ with probability $1$,
	\begin{align}
		\bE[\tau^0] \le \sum_{N \ge 1}N^2 \bP\left[J=N\right]. \label{Qbound}
	\end{align}

We now consider the distribution of $J$. If $J$ is at least $N,$ then by averaging one of the following must happen:
\begin{itemize}
\item[(i)] One of $S_N^R$ and $S_N^L$ is at least $N(p + (1/4 - p/2)).$
\item[(ii)] One of $E^L, E^R, {E}_{N}^L$, and ${E}_N^R$ is at least $N(1/8-p/4).$
\end{itemize}
Clearly $S_N^R$ and $S_N^L$ are both distributed like $\Bin(N,p)$ random variables and so, by Lemma \ref{feelthechern}, the probability that (i) occurs is at most $2e^{-(1/2-p)^2N/2}.$ On the other hand, by Lemma \ref{lem:thedescent}, we know that $E^L, E^R, {E}_N^L$ and ${E}_{N}^R$ are all distributed like $\Geom_{\ge 0}(\frac{1-2p}{1-p})$ random variables. If $X~\sim~\Geom_{\ge 0}(\frac{1-2p}{1-p})$, then $\bP\left[X \geq N(1/8-p/4)\right] \leq \bigl(1-\frac{1-2p}{1-p}\bigr)^{N(1/8-p/4)}$, and so the probability that (ii) occurs is at most
\[
4\bP\left[\Geom_{\ge 0}\left(\frac{1-2p}{1-p} \right) \ge N(1/8-p/4)\right] \leq 4\left(1-\frac{1-2p}{1-p}\right)^{N(1/8-p/4)}.
\]
Putting these together we see that for all $N \ge 1$ we have
\begin{align*} 
		\bP\left[J = N\right] \le \bP\left[J \ge N\right] & \le 2e^{-(1/2-p)^2N/2} + 4\biggl(1-\frac{1-2p}{1-p}\biggr)^{N(1/8-p/4)} \\
		& \le 2e^{-(1/2-p)^2N/2} + 4e^{-\frac{1-2p}{1-p}N(1/8-p/4)}.
\end{align*}
As the above bound on $\bP\left[J \ge N\right]$ tends to $0$ as $N \to \infty$, we see that $J < \infty$ almost surely. Hence, putting the obtained bound into \eqref{Qbound} gives
	\begin{align*}
		\bE[\tau^0] &\le \sum_{N \ge 1}N^2\left[2e^{-\frac{(1/2-p)^2}{2}N} + 4e^{-\frac{(1/2-p)^2}{1-p}N}\right]  \le 6\sum_{N \ge 1} N^2e^{-(1/2-p)^2N/2}.
	\end{align*}

This sum can be approximated by the integral $\int_{0}^{\infty}x^2e^{-(1/2-p)^2x/2}dx$. By a change of variables or by considering the pdf of a $\Gamma\left(3,(1/2-p)^2/2\right)$ random variable we get a bound of the form $O\left((1/2-p)^{-6}\right),$ as required.

\end{proof}

\section{Further questions}\label{yaymorewaffle}
There is still a gap between the upper and lower bounds in Theorem \ref{thm:p=1/2}. Following the conjecture presented in the seminar by \cite{Junge-seminar}, we also believe that the upper bound gives the right order $t^{3/4}.$

It would be interesting to know what happens in higher dimensions, where the problems seem to become more difficult and are likely to require additional ideas. It is also natural to ask what happens in other lattices: for example, are there analogous results to Theorems \ref{thm:p=1/2} and \ref{thm:p<1/2} that hold for the hexagonal lattice? We remark that \cite[Open Questions 1 and 2]{DGJLS} have conjectures here (which we believe to be true). Indeed recently \cite{JJLS} gave an elegant proof of a lower bound for the continuous time-based parking process. Their methods work for the discrete time-based parking process and should 
also work more generally in other settings.
Finally, what can we say for more general jump distributions? We conjecture that if the increments of the random walks $X^i$ on $\bZ$ are bounded, then Theorems \ref{thm:p=1/2} and \ref{thm:p<1/2} should still hold. Although similar methods could work, one would have to be careful about specifying parking places for cars (as in the parking strategy $T$ in Section~\ref{sec:p=1/2upper}) as cars might jump over them.

\bibliographystyle{amsplain}
\bibliography{parking}

\end{document}